\documentclass[preprint,10pt,3p,numbers,square]{elsarticle}

\usepackage{amssymb}
\usepackage{amsmath}
\usepackage{amsthm}

\usepackage{xfrac}

\usepackage{tikz}
\usepackage{comment}
\usepackage{hyperref}

\theoremstyle{definition}
\newtheorem{ass}{Assumption}

\usepackage{empheq}

\usepackage{array}    
\usepackage{makecell} 

\newcommand{\pupt}[1]{\frac{\partial {#1}}{\partial t}}

\newcommand{\dudt}[1]{\frac{d {#1}}{dt}}

\newcommand{\ddtu}[1]{\frac{d}{dt}\left( {#1} \right)}

\newcommand{\pupx}[1]{\frac{\partial {#1}}{\partial x}}

\newcommand{\bm}[1]{\boldsymbol{#1}}

\newdefinition{rmk}{Remark}
\newdefinition{dfn}{Definition}
\newdefinition{thm}{Theorem}
\newdefinition{cor}{Corollary}
\newdefinition{lem}{Lemma}
\newdefinition{prop}{Proposition}

\DeclareMathOperator*{\argmin}{arg\,min}

\usepackage{lineno}

\begin{document}

\begin{frontmatter}



\title{On the Existence of Pressure-Equilibrium-Preserving Numerical Fluxes for Supercritical Fluids}


\author[1,2]{R.B. Klein\corref{cor1}} 
\ead{rbk@cwi.nl}

\affiliation[1]{organization={Delft University of Technology, Process \& Energy},
            addressline={Leeghwaterstraat 39}, 
            city={Delft},
            postcode={2628 CB}, 
            state={Zuid-Holland},
            country={The Netherlands}}

\affiliation[2]{organization={Centrum Wiskunde \& Informatica, Scientific Computing},
            addressline={Science Park 123}, 
            city={Amsterdam},
            postcode={1098 XG}, 
            state={Noord-Holland},
            country={The Netherlands}}


\cortext[cor1]{Corresponding author}

\begin{abstract}
In this work we propose a new existence theorem for numerical-flux functions for supercritical fluids that are pressure-equilibrium-preserving (PEP). In particular, we characterize the existence of consistent numerical-flux functions that satisfy an algebraic PEP property. Our theory links the existence of PEP schemes to geometric properties of the equation of state describing the thermodynamics of the fluid. When these geometric properties fail for a pair of states on the same isobar, no PEP schemes of the considered form can exist on a domain containing those states. In our analysis the equation of state itself can be fully general only needing to satisfy some fundamental thermodynamic principles. Recently, PEP compatibility conditions for general equations of state have been derived \cite{channodal} that rely on the existence of certain thermodynamic derivatives which are not guaranteed to be defined under fundamental thermodynamic principles. The geometric conditions in our theory do not depend on the existence of these derivatives and recover the recent compatibility conditions in the case that these derivatives are defined. Finally, using numerical experiments we demonstrate for two supercritical fluids that our existence conditions are restrictive and thus that no PEP schemes of the form we consider exist on the domain we specify for these fluids. Using our geometric perspective, we also shed light on mechanisms by which PEP schemes can develop numerical issues, which we also demonstrate using numerical experiments.
\end{abstract}



\begin{keyword}
Structure preservation \sep Existence analysis \sep Supercritical fluids \sep Numerical-flux functions \sep Pressure-equilibrium preservation



\end{keyword}

\end{frontmatter}


\section{Introduction}
\subsection{Problem introduction}
Modern engineering applications increasingly operate at temperatures and pressures above the critical points of common working fluids. In this supercritical regime \cite{guardonenonideal}, thermodynamic behavior is strongly non-ideal and nonlinear, requiring real-gas equations of state (EOS). In particular, the disappearance of distinct liquid and gas phases allows rapid transitions between gas-like and liquid-like properties. In fluid flows, this can produce material interfaces with large density variations while pressure and temperature remain nearly constant \cite{Senguptafully}.

Numerical schemes frequently develop spurious pressure oscillations across such interfaces, which can be amplified by the nonlinear thermodynamics and destabilize simulations \cite{lacazecomparison}. This has motivated increasing interest in pressure-equilibrium-preserving (PEP) schemes, which discretely preserve constant-pressure and constant-velocity equilibrium solutions of the Euler equations and thereby prevent the generation of pressure oscillations across equilibrium material interfaces. 

Pressure-equilibrium preservation has been extensively studied for calorically perfect gases, particularly in the context of split-form and kinetic-energy-preserving discretizations \cite{shimapreventing, ranochapreventing, demichelenovel}. For multi-component and real-fluid flows, earlier approaches instead relied primarily on quasi-conservative, pressure-based, or double-flux formulations \cite{abgrallprevent, terashimaapproach, kawairobust, maentropy, lacazecomparison}, mostly trading exact conservation for improved pressure equilibrium. A quasi-conservative approach for real gases is given in \cite{baioscillation}. In \cite{bernadeskinetic, xucentral}, real-gas PEP formulations are derived by evolving pressure rather than total energy. More recent work has sought fully conservative approximately or exactly PEP discretizations for general equations of state \cite{terashimaapproximately, coppolapressure, degrendeleconstruction, channodal}.

However, even though there has been a considerable effort to derive PEP discretizations for general EOS, no proposed fully conservative PEP scheme so far is free of potential singularities. This naturally raises the question of whether schemes with these properties for real gases exist at all. A first step to answering this question was made in \cite{terashimaapproximately}, where a necessary PEP compatibility condition for numerical-flux functions satisfying \cite[Eq.14-16]{terashimaapproximately} was derived. This condition was recently revisited by \cite{channodal}. Noting the compatibility condition was analogous to Tadmor's entropy-conservation condition \cite{tadmorentropy, artianoaffordable}, they derived an algebraically equivalent compatibility condition \cite{channodal} to the one proposed in \cite{terashimaapproximately}. Both works rely on the assumption that a certain thermodynamic derivative exists. Essentially, this requires that the so-called Gr\"uneisen parameter determined by the EOS does not vanish. Although this is a very reasonable assumption for supercritical fluids \cite{mausbachcomparative}, this assumption is not based on any fundamental thermodynamic principle \cite{menikoffriemann} and therefore prevents the derived compatibility conditions from being fully general. More importantly, neither work provides a full existence theorem.

In this work, we will resolve these issues and provide a general existence characterization. To isolate the constraints imposed by equilibrium preservation from questions of flux regularity and the well-definedness of numerical schemes, we pose PEP as an algebraic property for numerical fluxes. Furthermore, rather than assuming the required derivatives exist as in \cite{channodal, terashimaapproximately}, we will adopt a novel geometric perspective removing the need for this assumption entirely. Based on fundamental thermodynamic principles alone, we are then able to derive necessary geometric conditions on a general real-gas EOS for the existence of consistent and algebraically PEP numerical-flux functions. When the thermodynamic derivatives assumed in \cite{channodal,terashimaapproximately} exist, our geometric conditions recover the compatibility condition derived in \cite{channodal}. Assuming the derived conditions are satisfied we can also construct a consistent and algebraically PEP numerical-flux function showing the conditions are also sufficient, finishing the characterization. 

The necessary and sufficient geometric conditions are easy to verify from a single plot of the EOS. For a few fluids and EOS we will demonstrate that there exist configurations of supercritical material gas-like/liquid-like interfaces for which the conditions do not hold and PEP schemes do not exist if their domain includes these interfaces. Our results will thus show that although it is possible to derive consistent numerical-flux functions with the algebraic PEP property for many pairs of points \cite{coppolapressure, degrendeleconstruction, channodal}, for many real-gas EOS, discretizations using these flux functions will inevitably contain singularities which cannot be removed if the domain is taken to be the full supercritical regime. However, even when PEP schemes for real gases have no singularities, the geometry of the EOS can still cause quite severe stability issues. Remarkably, these stability issues can occur both in under-resolved and even well-resolved simulations, showing that care is needed when using PEP schemes for supercritical fluids. We illustrate these issues using some simple resolved and under-resolved numerical experiments.

\subsection{Overview of main result}\label{ssec:overviewmainresult}
To state our main result rigorously a number of preliminary results have to be developed. To motivate these developments we will now provide an informal overview of our existence characterization and what steps have to be taken to achieve a rigorous statement. The characterization can be stated informally as follows, where we italicize important concepts:
\begin{quote}
    Let some EOS with a \emph{supercritical regime} be given. A consistent numerical-flux function with the \emph{algebraic PEP} property exists on this regime if and only if any two \emph{tangent lines} of any \emph{$(\rho,\varepsilon)$-isobar} have nonempty intersections. 
\end{quote}
Here, $(\rho,\varepsilon)$-isobars are assumed to be sufficiently smooth curves in $\mathbb{R}^2$ of constant pressure expressed in terms of the mass density $\rho$ and internal-energy density $\varepsilon$. To make this characterization rigorous we will proceed as follows:
\begin{enumerate}
    \item \emph{(Supercritical regime)} We will define how we model the supercritical regime and what conditions an EOS needs to satisfy so that such a regime exists. To keep our analysis as general as possible, we will provide a rather loose definition of supercritical regimes.
    \item \emph{(EOS geometry)} Our characterization links numerical-flux functions to isobars and their tangent lines, which follow from the EOS. We must show that these geometric objects are well-defined given the conditions on the EOS which we assume to hold.
    \item \emph{(Algebraic PEP)} To ensure our conditions can be necessary and sufficient using our proof approach, we define the PEP property in a specific way. Namely, as an algebraic property of the numerical flux. Therefore, we will clearly define and motivate this PEP formulation.
\end{enumerate}
The content of this work address these points as follows. In \autoref{sec:setting} we define the mathematical setting of our analysis. Here, we define supercritical regimes and show that the required EOS geometry is well-defined. In \autoref{sec:PEP} we introduce our formulation of the PEP property including some basic objects we will use repeatedly throughout the characterization. In \autoref{sec:characterization} we prove a rigorous formulation of our main result. We also compare our conditions to the compatibility conditions derived in \cite{channodal}. In \autoref{sec:experiments} we carry out some experiments to illustrate the stability issues of PEP schemes and that real fluids and EOS can violate the derived conditions and thus that no PEP schemes exist. Finally, we conclude in \autoref{sec:conclusion}.

\section{Assumptions and mathematical setting}\label{sec:setting}
\subsection{Equation of state}
Our existence characterization relates the geometry of the thermodynamic description of supercritical fluids to the existence of numerical fluxes with the pressure-equilibrium-preserving property. To establish this relation we make a number of fundamental assumptions on this thermodynamic description which we briefly outline now.

Specifically, we assume the thermodynamic state of fluids is described using an equation of state (EOS) given by a general thermodynamic potential in terms of its natural variables (see \cite{holystthermodynamics} for more details on natural variables). Such a formulation is particularly general as any well-defined EOS prescribing a particular thermodynamic quantity must be derived from a thermodynamic potential \cite{holystthermodynamics}. We will only assume some fundamental thermodynamic principles hold for this thermodynamic potential which we will adapt from \cite{menikoffriemann} to the setting of supercritical fluids. The particular thermodynamic potential used in \cite{menikoffriemann}, in which we will thus also formulate the assumptions underlying our setting, is the specific\footnote{Throughout, the adjective `specific' is to be read as `per unit mass'.} internal energy $e : \mathcal{T} \rightarrow \mathbb{R}$. Here, $\mathcal{T} \subset \mathbb{R}^2$ is the thermodynamic state space of the natural variables of the EOS $e$ and is defined as:
\begin{equation}
    \mathcal{T} := \{(\nu, \sigma) \in \mathbb{R}^2\, :\, \nu > 0,\, \sigma \geq \sigma_{\min} \},
    \label{eq:thermodynamicstatespace}
\end{equation}
for some $\sigma_{\min} \in \mathbb{R}$. Here, $\bm{\eta} := (\nu,\sigma) \in \mathcal{T}$ denotes the natural variables, where $\nu$ is referred to as the specific volume and $\sigma$ is the specific entropy. On the thermodynamic state space, $e$ is assumed to be $C^1$ and piecewise twice continuously differentiable \cite{menikoffriemann}. Many quantities of interest are defined as derivatives of $e$. For example, the first derivatives give pressure and temperature, respectively:
\begin{equation}
    p(\bm{\eta}) := -e_\nu(\bm{\eta}), \qquad T(\bm{\eta}) := e_\sigma(\bm{\eta}).
    \label{eq:pressuretemperature}
\end{equation}
For brevity, partial derivatives will be denoted by application of a subscript e.g. $e_{\nu}$ for the $\nu$-derivative of $e$. 

In our analysis, we will be specifically interested in a subset of states in $\mathcal{T}$ that are supercritical \cite{guardonenonideal}. Physically, this set can be defined\footnote{Strictly, we should also account for the fact that at sufficiently high pressures the supercritical fluid may solidify} as all $\bm{\eta} \in \mathcal{T}$ so that $p(\bm{\eta}) > p_{\mathrm{crit}}$ and $T(\bm{\eta}) > T_{\mathrm{crit}}$ where $p_{\mathrm{crit}},\,T_{\mathrm{crit}} \in \mathbb{R}_+$ are the critical pressure and temperature \cite{holystthermodynamics}, respectively. However, not all EOS posses a meaningful critical point ($p_{\mathrm{crit}},\,T_{\mathrm{crit}}$), while the pressure-equilibrium-preserving property can still be defined \cite{shimapreventing, ranochapreventing, demichelenovel}. For generality, we will therefore model the supercritical regime by any subset $\mathcal{S} \subset \mathcal{T}$ where the following general characteristics of supercritical thermodynamics hold.
\begin{dfn}[Admissible supercritical regime]  
    \label{dfn:admissiblesupercriticalregime}
    A nonempty subset $\mathcal{S} \subseteq \operatorname{interior}(\mathcal{T})$, open in $\mathbb{R}^2$, is an \emph{admissible supercritical regime} if the following conditions hold:
    \begin{enumerate}
        \item \emph{(Regularity)}: the internal energy is twice continuously differentiable $e \in C^2(\mathcal{S})$.
        \item \emph{(Monotonicity)}: for all $\bm{\eta} \in \mathcal{S}$ the pressure and temperature are positive:
        \begin{equation*}
            p(\bm{\eta}) > 0, \qquad T(\bm{\eta}) > 0,
        \end{equation*}
        and for each fixed $\nu$ the map $\sigma \mapsto e(\nu,\sigma)$ is injective on its domain in $\mathcal{S}$.
        \item \emph{(Strict stability)}: for all $\bm{\eta} \in \mathcal{S}$ it holds that:
        \begin{equation*}
            e_{\nu\nu}(\bm{\eta}),\ e_{\sigma\sigma}(\bm{\eta}) > 0, \qquad e_{\nu\nu}(\bm{\eta})e_{\sigma\sigma}(\bm{\eta}) > (e_{\nu\sigma}(\bm{\eta}))^2.
        \end{equation*}
    \end{enumerate}
\end{dfn}
The regularity and strict stability properties in this definition are designed to exclude phase transition regions where these properties are lost. This is also why we assume $\mathcal{S}$ is open, since phase transition regimes are usually closed \cite{holystthermodynamics} and thus $\mathcal{S}$ can be a natural complement to it. The monotonicity property, reflecting the monotonicity of $e$ in $(\nu,\sigma)$-coordinates, is a general property of thermodynamics away from absolute zero temperature. This is also why we assume $\mathcal{S}$ lies in the interior of $\mathcal{T}$ as entropy is minimal at the absolute zero isotherm. The injectivity assumption is made for our convenience and can often be shown. For many EOS this definition will certainly allow nonsupercritical states. This only serves to widen the applicability of our theorem. However, our main focus is to model supercritical fluids. This completes the list of assumptions we make on the thermodynamics for our main result.

\subsection{Density-energy coordinates}
The main thermodynamic assumptions behind our theorem are most naturally stated in $(\nu,\sigma)$-coordinates \cite{menikoffriemann}, however our theorem is most naturally stated using $(\rho,\varepsilon)$-coordinates. The use of differentiable $(\rho,\varepsilon)$-isobars in our theorem will require the pressure function to be $C^1$ in the new coordinates, therefore it is important to establish that this change-of-coordinates preserves the regularity of the EOS in the original $(\nu,\sigma)$-coordinates. Here, $\rho$ is the mass density and $\varepsilon$ is the internal energy density, thus we refer to the new coordinates as density-energy coordinates. We denote an admissible supercritical regime expressed in $\bm{\eta}$ as $\mathcal{S}_{\bm{\eta}}$ and as $\mathcal{S}_{\bm{\tau}}$ when it is expressed in $\bm{\tau} := (\rho,\varepsilon)$. The map between these coordinates is:
\begin{equation}
    \bm{\Phi} : \mathcal{S}_{\bm{\eta}} \rightarrow \mathbb{R}^2, \qquad (\nu,\sigma) \mapsto (\rho, \varepsilon) := \begin{bmatrix}
        \displaystyle\frac{1}{\nu} \\[0.3em] \displaystyle \frac{e(\nu,\sigma)}{\nu} 
    \end{bmatrix}.
    \label{eq:Phi}
\end{equation}
We define $\mathcal{S}_{\bm{\tau}} := \bm{\Phi}(\mathcal{S}_{\bm{\eta}})$. The regularity of functions on $\mathcal{S}_{\bm{\eta}}$ is preserved under $\bm{\Phi}$ if $\bm{\Phi}$ is a diffeomorphism which we show now.
\begin{prop}[$(\rho,\varepsilon)$-diffeomorphism]
    \label{prop:densityenergydiffeomorphism}
    For any admissible supercritical regime $\mathcal{S}_{\bm{\eta}}$, the change-of-coordinates map $\bm{\Phi} : \mathcal{S}_{\bm{\eta}} \rightarrow \mathbb{R}^2$ as in \eqref{eq:Phi} is a $C^2$ diffeomorphism onto its image $\mathcal{S}_{\bm{\tau}}$.
\end{prop}
The proof can be found in \autoref{ssec:densityenergydiffeomorphism}. Since $\bm{\Phi}$ is a diffeomorphism we will refer to any image of an admissible supercritical regime under $\bm{\Phi}$ also as an admissible supercritical regime. 

\subsection{Pressure and isobars}
In our theorem we make use of a number of geometric objects related to the pressure, i.e.\ isobars and their tangent lines, an impression of these is provided in \autoref{fig:overview}. We will now show that our assumptions in \autoref{dfn:admissiblesupercriticalregime} of an admissible supercritical regime are sufficient to ensure these objects are well-defined. These objects are expressed in $\mathcal{S}_{\bm{\tau}}$, by some abuse of notation, we will therefore redefine pressure in terms of these coordinates as:
\begin{equation*}
    p(\bm{\tau}) := p(\bm{\Phi}^{-1}(\bm{\tau})).
\end{equation*}
We can use the regularity of the diffeomorphism $\bm{\Phi}$ that we have demonstrated and the strict stability of admissible supercritical regimes to show that $p(\bm{\tau})$ indeed has the following essential property. 
\begin{prop}[Pressure regularity]
    \label{prop:pressureregularity}
    For any admissible supercritical regime $\mathcal{S}_{\bm{\tau}}$ the pressure satisfies $p \in C^1(\mathcal{S}_{\bm{\tau}})$ and its gradient is nonvanishing:
    \begin{equation*}
        \Vert\nabla p(\bm{\tau})\rVert > 0, \qquad \forall \bm{\tau} \in \mathcal{S}_{\bm{\tau}}.
    \end{equation*}
\end{prop}
The proof can be found in \autoref{ssec:pressureregularity}. The isobars in our theorem are level sets of pressure in $\mathcal{S}_{\bm{\tau}}$ defined as:
\begin{equation*}
    \Gamma_P := \{ \bm{\tau} \in \mathcal{S}_{\bm{\tau}}\, :\, p(\bm{\tau}) = P\}.
\end{equation*}
Our theorem assumes $\Gamma_P$ are one-dimensional differentiable curves. It is therefore important that our assumptions in \autoref{dfn:admissiblesupercriticalregime} are sufficient to ensure this. This can be done using \autoref{prop:pressureregularity} and the implicit-function theorem.
\begin{prop}[Isobar curves]
    \label{prop:isobarcurves}
    For any admissible supercritical regime $\mathcal{S}_{\bm{\tau}}$ and any point on a nonempty isobar $\Gamma_P$ on $\mathcal{S}_{\bm{\tau}}$ denoted by $\bm{\tau}_0 \in \Gamma_P$ there exist open intervals $I,J,K \subset \mathbb{R}$ with $J \times K \subset \mathcal{S}_{\bm{\tau}}$ and a regular $C^1$ diffeomorphism defined as:
    \begin{equation*}
        \bm{\varphi} : I \rightarrow \Gamma_P \cap (J \times K),
    \end{equation*}
    with $\bm{\varphi}(t_0) = \bm{\tau}_0$ for some $t_0 \in I$.
\end{prop}
The proof is found in \autoref{ssec:isobarcurves}. For a regular local parameterization $\bm{\varphi}$ with $\bm{\varphi}(t) = \bm{\tau}$, we can define the tangent space as:
\begin{equation*}
    T_{\bm{\tau}}\Gamma_P := \operatorname{span}(\{\bm{\varphi}'(t)\}).
\end{equation*}
This definition is independent of the chosen regular parameterization. The tangent space is a vector space, by shifting its origin we can finally define tangent lines in terms of which our characterization is formulated.
\begin{dfn}[Tangent lines]
    For any $P \in \mathbb{R}_+$ so that the isobar $\Gamma_P$ on some admissible supercritical regime $\mathcal{S}_{\bm{\tau}}$ is nonempty, a \emph{tangent line} at some point $\bm{\tau}\in\Gamma_P$ is denoted as $L_{\bm{\tau}}\Gamma_P$ and defined as:
    \begin{equation*}
        L_{\bm{\tau}} \Gamma_P := \bm{\tau} + T_{\bm{\tau}}\Gamma_P.
    \end{equation*}
\end{dfn}
Note that $L_{\bm{\tau}} \Gamma_P$ is no longer necessarily a vector subspace, but an affine space, since $\bm{0} \in L_{\bm{\tau}} \Gamma_P$ may not hold. Since tangent lines lie tangent to isobars, the tangent vectors constituting these tangent lines are orthogonal to the pressure gradient which we state as a lemma for later reference.
\begin{lem}[Tangent-space orthogonality]
    \label{lem:tangentspaceorthogonality}
    For any $P \in \mathbb{R}_+$ so that the isobar $\Gamma_P$ on some admissible supercritical regime $\mathcal{S}_{\bm{\tau}}$ is nonempty, the tangent space $T_{\bm{\tau}}\Gamma_P$ at a point $\bm{\tau} \in \Gamma_P$ satisfies:
    \begin{equation}
        T_{\bm{\tau}} \Gamma_P = \operatorname{span}\left(\left\{\begin{bmatrix}
            -p_{\varepsilon}(\bm{\tau}) \\[0.3em] p_{\rho}(\bm{\tau})
        \end{bmatrix}\right\} \right) = \ker\left(\left<\nabla p(\bm{\tau}),\cdot \right>\right).
        \label{eq:tangentspaceidentity}
    \end{equation}
\end{lem}
The proof is given in \autoref{ssec:tangentspaceorthogonality}. We have thus shown that the assumptions defining the admissible supercritical regimes in \autoref{dfn:admissiblesupercriticalregime} are sufficient to ensure the geometric objects in our main theorem are always well-defined.

\begin{figure}
    \centering
    \includegraphics[width=0.6\linewidth]{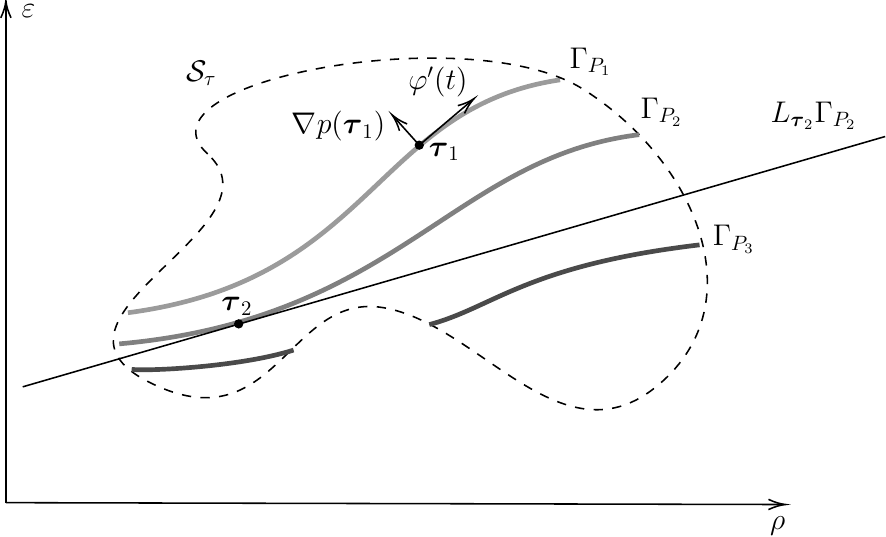}
    \caption{An impression of an admissible supercritical regime $\mathcal{S}_{\bm{\tau}}$ in $(\rho,\varepsilon)$-space along with three isobars $\Gamma_{P_1}$, $\Gamma_{P_2}$ and $\Gamma_{P_3}$ indicated by different shades of gray. At some point $\bm{\tau}_1 \in \Gamma_{P_1}$, the parameterization derivative $\bm{\varphi}'(t)$ and the pressure gradient $\nabla p(\bm{\tau}_1)$ are sketched and are orthogonal. At some point $\bm{\tau}_2 \in \Gamma_{P_2}$, the tangent line $L_{\bm{\tau}_2}\Gamma_{P_2}$ is sketched. The isobar $\Gamma_{P_3}$ consists of multiple connected components.}
    \label{fig:overview}
\end{figure}
\section{Pressure-equilibrium-preserving schemes}\label{sec:PEP}
In the previous section we established the isobar geometry used in our existence characterization. In this section we will define the particular form of the pressure-equilibrium-preservation (PEP) property that we will characterize and show how it can be formulated using this geometry. 

\subsection{The Euler equations}
We study numerical fluxes for the conservative discretization of the one-dimensional Euler equations in conservative variables $\bm{U} := (\rho,\, m,\, E)$. Here, $\rho$ is the same mass density as earlier, $m$ is the momentum density and $E$ the total-energy density. The one-dimensional Euler equations are then stated as:
\begin{equation}
    \pupt{\bm{U}} + \pupx{\bm{f}(\bm{U})} = 0, \qquad \bm{f}(\bm{U}) = \begin{bmatrix}
        \rho v \\ \rho v^2 + p \\ v(E + p)
    \end{bmatrix},
    \label{eq:Euler}
\end{equation}
where $v$ is the fluid velocity and $p$ is the thermodynamic pressure as discussed in detail in the previous section. To connect these variables to the thermodynamic description in the previous section, we recover the velocity and internal-energy density through:
\begin{equation}
    v(\bm{U}) = \frac{m}{\rho}, \qquad \varepsilon(\bm{U}) = E - \frac12 \frac{m^2}{\rho}.
    \label{eq:vpdefinition}
\end{equation}
The corresponding thermodynamic state is $\bm{\tau}(\bm{U}) := (\rho, \varepsilon(\bm{U}))$ and the pressure in \eqref{eq:Euler} can be evaluated as $p(\bm{\tau}(\bm{U}))$.

\subsection{Supercritical conservative variables}
The numerical-flux functions we study are defined on conservative variables. To understand whether a numerical flux with the PEP property exists when the corresponding thermodynamic states $\bm{\tau}(\bm{U})$ are in some prescribed admissible supercritical regime $\mathcal{S}_{\bm{\tau}}$, we must specify the corresponding set of admissible conservative states. Each thermodynamic state $\bm{\tau} = (\rho, \varepsilon) \in \mathcal{S}_{\bm{\tau}}$, together with an arbitrary velocity $v \in \mathbb{R}$, determines a conservative state through the map:
\begin{equation}
    \bm{\Theta} : \mathcal{S}_{\bm{\tau}} \times \mathbb{R} \rightarrow \mathbb{R}^3, \qquad (\bm{\tau}, v) \mapsto \begin{bmatrix}
        \rho \\ \rho v \\ \varepsilon + \frac12 \rho v^2
    \end{bmatrix}.
    \label{eq:conservativediffeomorphism}
\end{equation}
We define the associated conservative-variable domain by $\mathcal{S}_{\bm{U}} := \bm{\Theta}(\mathcal{S}_{\bm{\tau}}\times \mathbb{R})$. The following now holds.
\begin{prop}[Supercritical conservative variables]
    \label{prop:supercriticalconservativevariables}
    For any admissible supercritical regime $\mathcal{S}_{\bm{\tau}}$, the map $\bm{\Theta} : \mathcal{S}_{\bm{\tau}} \times \mathbb{R} \rightarrow \mathcal{S}_{\bm{U}}$ is a smooth diffeomorphism and the image $\mathcal{S}_{\bm{U}}$ satisfies:
    \begin{equation*}
        \mathcal{S}_{\bm{U}} = \{(\rho,m,E)\in \mathbb{R}^3\, :\, \rho > 0,\, \bm{\tau}(\rho,m,E) \in \mathcal{S}_{\bm{\tau}} \}.
    \end{equation*}
\end{prop}
The proof is given in \autoref{ssec:supercriticalconservativevariables}.

\subsection{Pressure-velocity equilibrium}
With the appropriate domains defined, we can now introduce the pressure-velocity equilibria. These are solutions of the Euler equation where pressure and velocity stay constant in time. For the Euler equations \eqref{eq:Euler} these can be obtained for solution profiles with spatially constant pressure and velocity, while mass density and internal-energy density may vary. Therefore, for some $P \in \mathbb{R}_+$ and $V \in \mathbb{R}$, the solution profile at any point in the domain\footnote{Or on some neighborhood of interest} should be an element of the $(P,V)$-equilibrium set:
\begin{equation*}
    \mathcal{M}_{P,V} := \{\bm{U} \in \mathcal{S}_{\bm{U}}\,:\, p(\bm{\tau}(\bm{U})) = P,\, v(\bm{U}) = V\}.
\end{equation*}
The fact that smooth solutions satisfying $\bm{U}(x,t) \in \mathcal{M}_{P,V}$ lead to $(P,V)$-equilibria can be seen from the Euler flux:
\begin{equation*}
    \bm{f}(\bm{U}) = \begin{bmatrix}
        \rho V \\ \rho V^2 + P \\ V(E + P)
    \end{bmatrix} = V \bm{U} +\begin{bmatrix}
        0 \\ P \\ VP
    \end{bmatrix}, \qquad \forall \bm{U}\in \mathcal{M}_{P,V}.
\end{equation*}
When $\bm{U}(x,t) \in \mathcal{M}_{P,V}$ for all $x$ in the domain at some fixed $t$, the Euler equations \eqref{eq:Euler} then reduce to:
\begin{equation*}
    \pupt{\bm{U}} + V \pupx{\bm{U}} = 0.
\end{equation*}
The solution profiles, including the pressure and velocity, are advected at velocity $V$. Since the pressure and velocity are spatially constant, they thus stay constant in time. These solution configurations can also occur locally in space, in which case the pressure and velocity remain constant only locally in time. We will analyze conservative discretizations of the Euler equations \eqref{eq:Euler} that reproduce this property.

The connection to the geometry of the previous section is that fixing pressure restricts the thermodynamic states to an isobar $\Gamma_P$, while fixing velocity determines their corresponding conservative states. The following proposition makes this precise.
\begin{prop}[Equilibrium set]
    \label{prop:equilibriumset}
    For some $P\in \mathbb{R}_+$ and $V \in \mathbb{R}$ and an isobar $\Gamma_P$ on an admissible supercritical regime $\mathcal{S}_{\bm{\tau}}$, the $(P,V)$-equilibrium set satisfies:
    \begin{equation*}
        \mathcal{M}_{P,V} = \{ \bm{\Theta}(\bm{\tau},V)\, :\, \bm{\tau}\in \Gamma_P\}.
    \end{equation*}
\end{prop}
The proof can be found in \autoref{ssec:equilibriumset}. 

\subsection{PEP schemes and algebraic PEP}\label{ssec:algebraicpep}
With this understanding of $(P,V)$-equilibria, we can provide our working definition of the PEP property. We begin by defining the type of discretization of the Euler equation \eqref{eq:Euler} we are looking at. 
We are concerned with conservative numerical schemes to solve the
Euler equations \eqref{eq:Euler} on a grid with grid-cell size $h>0$.
In particular, we are interested in the spatial discretization of
\eqref{eq:Euler} and will not consider effects of time discretization. We consider such a spatial discretization to be a \emph{numerical scheme} if for any admissible initial grid state, the resulting system of ordinary differential equations (ODEs) admits a local classical solution\footnote{A continuously
differentiable solution satisfying the ODEs pointwise.}. We assume these schemes are constructed from 
two-point numerical-flux functions
$\bm{f}_h:\mathcal{S}_{\bm{U}}\times\mathcal{S}_{\bm{U}}\rightarrow\mathbb{R}^3$,
where $\bm{f}_h$ is consistent with the Euler flux in the following
sense.
\begin{dfn}[Consistency]
    \label{dfn:consistency}
    On an admissible supercritical regime $\mathcal{S}_{\bm{U}}$, a numerical-flux 
    function
    $\bm{f}_h:\mathcal{S}_{\bm{U}}\times\mathcal{S}_{\bm{U}}\rightarrow\mathbb{R}^3$
    is \emph{consistent} with the Euler flux if:
    \begin{equation*}
        \bm{f}_h(\bm{U},\bm{U})=\bm{f}(\bm{U}),
        \qquad \forall \bm{U}\in\mathcal{S}_{\bm{U}}.
    \end{equation*}
\end{dfn}
Note furthermore that we have assumed throughout that the spatial
domain is one-dimensional. We made
this choice because extensions to multiple dimensions are mostly trivial while requiring extra
notation.  Finally,
we consider a simple form of a conservative scheme, using a
difference of numerical-flux functions to discretize the flux
divergence in an interior grid cell:
\begin{equation}
    h \dudt{\bm{U}_C}
    +\bm{f}_h(\bm{U}_R,\bm{U}_C)
    -\bm{f}_h(\bm{U}_C,\bm{U}_L)=0,
    \label{eq:discretization}
\end{equation}
where $\bm{U}_C\in\mathcal{S}_{\bm{U}}$ is the grid value in the cell, with
left and right neighboring grid values
$\bm{U}_L,\bm{U}_R\in\mathcal{S}_{\bm{U}}$, respectively.

A scheme of the form \eqref{eq:discretization} is said to be
pressure-equilibrium-preserving (PEP) if pressure and velocity time derivatives vanish in some cell when locally the grid values are in the same equilibrium set. 
\begin{dfn}[Pressure-equilibrium preservation]
    \label{dfn:pressureequilibriumpreservation}
    A scheme of the form \eqref{eq:discretization} with a consistent numerical flux $\bm{f}_h : \mathcal{S}_{\bm{U}} \times \mathcal{S}_{\bm{U}} \rightarrow \mathbb{R}^3$, where $\mathcal{S}_{\bm{U}}$ is an admissible supercritical regime, is said to be pressure-equilibrium-preserving (PEP) if for any admissible initial conditions and time $t$ for which the scheme is defined and any $P \in \mathbb{R}_+$ and $V \in \mathbb{R}$ the following implication holds:
    \begin{equation*}
        \bm{U}_R(t),\bm{U}_C(t),\bm{U}_L(t) \in \mathcal{M}_{P,V} \qquad \implies \qquad \dudt{p(\bm{\tau}(\bm{U}_C))} = 0, \quad \dudt{v(\bm{U}_C)} = 0.
    \end{equation*}
\end{dfn}
It is important to note that the use of time derivatives in this definition implicitly implies that solutions to the set of ODEs defined by the semi-discretization exist and are differentiable. This can be ensured by classical local ODE existence theory such as the Picard-Lindel\"of theorem \cite{hartmanordinary}. However, application of this theory requires the numerical-flux function to be continuous. Including this regularity condition into the characterization would likely make the analysis significantly more involved. To prevent the need to consider this technical issue, we instead rephrase the PEP condition from a time-derivative-based formulation to a weaker algebraic formulation. This formulation does not require any regularity of the numerical flux. We derive it as follows. Assume that a PEP scheme exists and that, at some time $t$, $\bm{U}_R(t),\bm{U}_C(t),\bm{U}_L(t) \in \mathcal{M}_{P,V}$ and \eqref{eq:discretization} hold. The chain rule and \eqref{eq:discretization} then give:
\begin{align*}
    0=\dudt{v(\bm{U}_C)}
    &=\nabla v(\bm{U}_C)^T\dudt{\bm{U}_C}\\
    &=-\frac1h\nabla v(\bm{U}_C)^T
    \left[
        \bm{f}_h(\bm{U}_R,\bm{U}_C)
        -\bm{f}_h(\bm{U}_C,\bm{U}_L)
    \right],
\end{align*}
and, denoting the Jacobian matrix operator as $D$:
\begin{align*}
    0=\dudt{p(\bm{\tau}(\bm{U}_C))}
    &=\nabla p(\bm{\tau}(\bm{U}_C))^T
      D\bm{\tau}(\bm{U}_C)\dudt{\bm{U}_C}\\
    &=-\frac1h\nabla p(\bm{\tau}(\bm{U}_C))^T
      D\bm{\tau}(\bm{U}_C)
    \left[
        \bm{f}_h(\bm{U}_R,\bm{U}_C)
        -\bm{f}_h(\bm{U}_C,\bm{U}_L)
    \right].
\end{align*}
Since the initial equilibrium configuration can be chosen
arbitrarily, the resulting flux identities must hold for every
triple $\bm{U}_R,\bm{U}_C,\bm{U}_L \in \mathcal{M}_{P,V}$. Note that the equalities on the second lines involve only
values of the numerical flux, and can therefore be formulated
without assuming either its regularity or the existence of a
time-dependent solution and a numerical scheme. For some two-point function $\bm{a} : \mathcal{S}_{\bm{U}} \times \mathcal{S}_{\bm{U}} \rightarrow \mathbb{R}^k$ with $k \in \mathbb{N}$, denote $\Delta \bm{a}(\bm{U}_R,\bm{U}_C, \bm{U}_L) := \bm{a}(\bm{U}_R,\bm{U}_C)- \bm{a}(\bm{U}_C,\bm{U}_L)$. This then motivates the following definition.
\begin{dfn}[Algebraic pressure-equilibrium preservation]
    \label{dfn:algebraicpressureequilibriumpreservation}
    A consistent numerical flux $\bm{f}_h : \mathcal{S}_{\bm{U}} \times \mathcal{S}_{\bm{U}} 
    \rightarrow \mathbb{R}^3$, where $\mathcal{S}_{\bm{U}}$ is an admissible supercritical regime, is said to be \emph{algebraically pressure-equilibrium-preserving} if, for every $P\in\mathbb{R}_+$ and $V\in\mathbb{R}$:
    \begin{equation*}
        \bm{U}_R,\bm{U}_C,\bm{U}_L\in\mathcal{M}_{P,V},
    \end{equation*}
    implies:
    \begin{equation}
        \nabla v(\bm{U}_C)^T \Delta \bm{f}_h(\bm{U}_R,\bm{U}_C,\bm{U}_L)=0, \qquad \nabla p(\bm{\tau}(\bm{U}_C))^T D\bm{\tau}(\bm{U}_C) \Delta \bm{f}_h(\bm{U}_R,\bm{U}_C,\bm{U}_L)=0.
        \label{eq:apepconditions}
    \end{equation}
\end{dfn}
This will be the PEP formulation that we will characterize.

\section{Existence characterization of algebraic PEP fluxes}\label{sec:characterization}
\subsection{Necessary conditions}
We can now begin proving our existence characterization. We will first derive necessary geometric conditions on the EOS for the existence of a consistent algebraic PEP flux. As already stated in \autoref{ssec:overviewmainresult}, these geometric conditions entail that any two tangent lines of an isobar must always intersect if a numerical flux is algebraically PEP given that EOS. This will be carried out in two main steps. First, we will find a relation between the mass flux $f_h^\rho$ and momentum flux $f_h^m$ that must hold between states in the same $(P,V)$-equilibrium set $\mathcal{M}_{P,V}$ if a numerical flux satisfies the algebraic velocity condition in \eqref{eq:apepconditions}. Second, using this flux relation, consistency and the fact that the pressure condition in \eqref{eq:apepconditions} is also satisfied, we will formulate the final necessary condition. In our analysis, we will make use of the fact that the algebraic PEP property must hold for \emph{any} triple $\bm{U}_R, \bm{U}_C, \bm{U}_L \in \mathcal{M}_{P,V}$ in some $(P,V)$-equilibrium set. Using this fact and consistency, we can simplify the three-point algebraic PEP conditions based on the three-point spatial discretization \eqref{eq:discretization} to conditions including only a single arbitrary pair of points in $\mathcal{M}_{P,V}$ by making useful choices of triples. This technique has recently also been used to characterize generalized entropy-conservation conditions for conservative and nonconservative systems \cite{artianoaffordable} and to characterize the PEP conditions in \cite{channodal} in a similar manner to here.

First, we start with analyzing the velocity-equilibrium condition. We will prove the following result.
\begin{prop}[Mass-momentum-flux consistency]
    \label{prop:massmomentumfluxconsistency}
    Let $\bm{f}_h : \mathcal{S}_{\bm{U}} \times \mathcal{S}_{\bm{U}} \rightarrow \mathbb{R}^3$ be a consistent, algebraically PEP numerical flux, where $\mathcal{S}_{\bm{U}}$ is an admissible supercritical regime. Then for any $P \in \mathbb{R}_+$ and $V \in \mathbb{R}$ the following consistency relation holds:
    \begin{equation}
        f_h^m(\bm{U}_R,\bm{U}_L) = V f_h^{\rho}(\bm{U}_R,\bm{U}_L) + P, \qquad \forall \bm{U}_R,\bm{U}_L \in \mathcal{M}_{P,V}.
        \label{eq:massmomentumfluxrelation}
    \end{equation}
\end{prop}
This result is a direct consequence of the velocity equilibrium.
\begin{proof}
    Assume the numerical flux $\bm{f}_h$ is algebraically PEP and consistent and $\mathcal{S}_{\bm{U}}$ is some admissible supercritical regime. Let $P\in\mathbb{R}_+$ and $V \in \mathbb{R}$ be arbitrary. If $P,V$ are such that $\mathcal{M}_{P,V}$ is empty \eqref{eq:massmomentumfluxrelation} is vacuous and holds automatically. Thus we can assume that $P,V$ are such that $\mathcal{M}_{P,V}$ is nonempty. Note that $v$ is differentiable on $\mathcal{S}_{\bm{U}}$, its gradient satisfies:
    \begin{equation*}
        \nabla v(\bm{U}) = \begin{bmatrix}
            -\frac{V}{\rho} \\ \frac{1}{\rho} \\ 0
        \end{bmatrix}, \qquad \forall \bm{U} \in \mathcal{M}_{P,V},
    \end{equation*}
    using \eqref{eq:vpdefinition} and that $v(\bm{U}) = V$ for $\bm{U}\in\mathcal{M}_{P,V}$. By \autoref{dfn:algebraicpressureequilibriumpreservation} of the algebraic PEP property, we have for any arbitrary $\bm{U}_R, \bm{U}_C, \bm{U}_L\in\mathcal{M}_{P,V}$ that:
    \begin{align*}
        0&=\nabla v(\bm{U}_C)^T\Delta \bm{f}_h(\bm{U}_R,\bm{U}_C,\bm{U}_L)  \\
        &=-\frac{V}{\rho_C}\left[f_h^\rho(\bm{U}_R, \bm{U}_C) - f_h^\rho(\bm{U}_C, \bm{U}_L) \right] + \frac{1}{\rho_C}\left[f_h^m(\bm{U}_R, \bm{U}_C) - f_h^m(\bm{U}_C, \bm{U}_L) \right] .
    \end{align*}
    Choosing $\bm{U}_C = \bm{U}_L$, it holds that:
    \begin{equation*}
        0 = -\frac{V}{\rho_L}\left[f_h^\rho(\bm{U}_R,\bm{U}_L) - \rho_L V \right] + \frac{1}{\rho_L}\left[f_h^m(\bm{U}_R,\bm{U}_L) - (\rho_L V^2 +P) \right].
    \end{equation*}
    Simplifying gives:
    \begin{equation*}
        f_h^m(\bm{U}_R,\bm{U}_L) = Vf_h^\rho(\bm{U}_R,\bm{U}_L) + P.
    \end{equation*}
    The result then follows since $\bm{U}_R,\bm{U}_L \in \mathcal{M}_{P,V}$ are arbitrary.
\end{proof}

Before moving on to the full necessary conditions by analyzing the pressure condition in \eqref{eq:apepconditions} we will prove an intermediate result. We will show that an algebraically PEP numerical flux and the gradient of the internal-energy density $\varepsilon$ satisfy a certain relation on the $(P,V)$-equilibrium set. Since pressure is a function of $\bm{\tau} = (\rho,\varepsilon)$ this will simplify the analysis of the pressure equilibrium as we can then work with two variables given by the $(\rho,\varepsilon)$-coordinates instead of conservative variables $\bm{U}$ with three components.

\begin{prop}[Internal-energy condition]
    \label{prop:energyevolution}
    Let $\bm{f}_h : \mathcal{S}_{\bm{U}} \times \mathcal{S}_{\bm{U}} \rightarrow \mathbb{R}^3$ be a consistent, algebraically PEP numerical flux, where $\mathcal{S}_{\bm{U}}$ is an admissible supercritical regime. Then for any $P \in \mathbb{R}_+$ and $V \in \mathbb{R}$ the following is satisfied:
    \begin{equation*}
        \nabla \varepsilon(\bm{U}_C)^T\Delta \bm{f}_h(\bm{U}_R,\bm{U}_C,\bm{U}_L) = f_h^\varepsilon(\bm{U}_R, \bm{U}_C; V) - f_h^\varepsilon(\bm{U}_C, \bm{U}_L;V), \qquad \forall \bm{U}_R, \bm{U}_C, \bm{U}_L \in \mathcal{M}_{P,V},
    \end{equation*}
    where:
    \begin{equation*}
        f_h^\varepsilon(\bm{U}_R, \bm{U}_C; V) := f_h^E(\bm{U}_R, \bm{U}_C) - \frac12 V^2 f_h^\rho(\bm{U}_R, \bm{U}_C).
    \end{equation*}
\end{prop}
\begin{proof}
    Assume the numerical flux $\bm{f}_h$ is algebraically PEP and consistent and $\mathcal{S}_{\bm{U}}$ is some admissible supercritical regime. Let $P\in \mathbb{R}_+$ and $V \in \mathbb{R}$ be arbitrary such that $\mathcal{M}_{P,V}$ is nonempty. Note that $\varepsilon$ is differentiable on $\mathcal{S}_{\bm{U}}$, its gradient satisfies:
    \begin{equation*}
        \nabla \varepsilon(\bm{U}) = \begin{bmatrix}
            \frac12 V^2 \\ -V \\ 1
        \end{bmatrix}, \qquad \forall \bm{U} \in \mathcal{M}_{P,V}.
    \end{equation*}
    Let $\bm{U}_R, \bm{U}_C, \bm{U}_L \in \mathcal{M}_{P,V}$ also be arbitrary. Using the gradient on $\mathcal{M}_{P,V}$, it follows by definition that:
    \begin{align*}
        \nabla \varepsilon(\bm{U}_C)^T\Delta \bm{f}_h(\bm{U}_R,\bm{U}_C,\bm{U}_L) &= \left[f_h^E(\bm{U}_R, \bm{U}_C) - f_h^E(\bm{U}_C, \bm{U}_L) \right] -V\left[f_h^m(\bm{U}_R, \bm{U}_C) - f_h^m(\bm{U}_C, \bm{U}_L) \right]  \\
        &\qquad \qquad  + \frac12 V^2\left[f_h^\rho(\bm{U}_R, \bm{U}_C) - f_h^\rho(\bm{U}_C, \bm{U}_L) \right].
    \end{align*}
    Since $\bm{U}_R, \bm{U}_C, \bm{U}_L \in \mathcal{M}_{P,V}$ and the numerical flux is algebraically PEP and consistent, \autoref{prop:massmomentumfluxconsistency} applies. We can therefore simplify:
    \begin{align*}
        \nabla \varepsilon(\bm{U}_C)^T\Delta \bm{f}_h(\bm{U}_R,\bm{U}_C,\bm{U}_L) &= \left[f_h^E(\bm{U}_R, \bm{U}_C) - f_h^E(\bm{U}_C, \bm{U}_L) \right] - \frac12 V^2\left[f_h^\rho(\bm{U}_R, \bm{U}_C) - f_h^\rho(\bm{U}_C, \bm{U}_L) \right]  \\
        &= \left[f_h^E(\bm{U}_R, \bm{U}_C) - \frac12 V^2 f_h^\rho(\bm{U}_R, \bm{U}_C)\right] - \left[f_h^E(\bm{U}_C, \bm{U}_L)  - \frac12 V^2f_h^\rho(\bm{U}_C, \bm{U}_L) \right]  \\
        &= f_h^\varepsilon(\bm{U}_R, \bm{U}_C; V) - f_h^\varepsilon(\bm{U}_C, \bm{U}_L;V).
    \end{align*}
    The result follows since $\bm{U}_R, \bm{U}_C, \bm{U}_L \in \mathcal{M}_{P,V}$ are arbitrary.
\end{proof}

With this intermediate result established we can provide the final necessity lemma of this section. This lemma shows that a consistent numerical flux can only be algebraically PEP on all triples of $\mathcal{S}_{\bm{U}} \times \mathcal{S}_{\bm{U}}\times \mathcal{S}_{\bm{U}}$ if the EOS satisfies certain strict geometric conditions. Specifically, any isobar $\Gamma_P$ should have tangent lines $L_{\bm{\tau}}\Gamma_P$ that intersect any other tangent lines of $\Gamma_P$ in at least one point. This rules out EOS with isobars in $(\rho,\varepsilon)$-coordinates containing an inflection as sketched in \autoref{fig:paraiso}. Such an inflection tends to represent competition between different nonlinear terms in a real-gas EOS and is usually not a special thermodynamic locus like the critical point or the Widom line, which marks gas-like to liquid-like transitions \cite{Senguptafully, guardonenonideal}. To illustrate, a brief physical discussion of isobar inflections of a Van-der-Waals fluid \cite{vanderwaalscontinuiteit} is provided in \autoref{ssec:physicalinflection}. For many nonlinear EOS for supercritical fluids this geometric condition is thus generally \emph{not} satisfied and therefore the lemma shows that the singularities found for example in the recent PEP flux of \cite{coppolapressure} are unavoidable. 

\begin{lem}[Necessary isobar geometry]
    \label{lem:necessaryisobargeometry}
    Let $\bm{f}_h : \mathcal{S}_{\bm{U}} \times \mathcal{S}_{\bm{U}} \rightarrow \mathbb{R}^3$ be a consistent, algebraically PEP numerical flux, where $\mathcal{S}_{\bm{U}}$ is an admissible supercritical regime. Then for any $P \in \mathbb{R}_+$ so that the isobar $\Gamma_P$ is nonempty, it holds that:
    \begin{equation*}
        L_{\bm{\tau}_1}\Gamma_P \cap L_{\bm{\tau}_2}\Gamma_P \neq \varnothing, \qquad \forall \bm{\tau}_1,\bm{\tau}_2 \in \Gamma_P.
    \end{equation*}
\end{lem}

\begin{proof}
    Assume the numerical flux $\bm{f}_h$ is algebraically PEP and consistent and $\mathcal{S}_{\bm{U}}$ is some admissible supercritical regime. Let $P\in \mathbb{R}_+$ be arbitrary so that $\Gamma_P$ is nonempty. Also, choose any $V \in \mathbb{R} \setminus \{0\}$. By \autoref{prop:pressureregularity}, pressure is $C^1$ on the admissible supercritical regime. Note that:
    \begin{equation*}
        D\bm{\tau}(\bm{U}) = \begin{bmatrix}
            \nabla \rho(\bm{U})^T \\ \nabla \varepsilon(\bm{U})^T
        \end{bmatrix}, \qquad \nabla \rho(\bm{U}) = \begin{bmatrix}
            1 \\ 0 \\ 0
        \end{bmatrix}.
    \end{equation*}
    By \autoref{dfn:algebraicpressureequilibriumpreservation} of the algebraic PEP property, it holds for all $\bm{U}_R, \bm{U}_C, \bm{U}_L\in \mathcal{M}_{P,V}$ that, denoting $\bm{\tau}_C := \bm{\tau}(\bm{U}_C)$:
    \begin{align*}
        0 &= \nabla {p(\bm{\tau}_C)}^TD\bm{\tau}(\bm{U}_C) \Delta \bm{f}_h(\bm{U}_R,\bm{U}_C,\bm{U}_L) \\
        &= p_{\rho}(\bm{\tau}_C)\left[f_h^\rho(\bm{U}_R, \bm{U}_C) - f_h^\rho(\bm{U}_C, \bm{U}_L) \right] + p_{\varepsilon}(\bm{\tau}_C)\left[f_h^\varepsilon(\bm{U}_R, \bm{U}_C;V) - f_h^\varepsilon(\bm{U}_C, \bm{U}_L;V) \right],
    \end{align*}
    where we used \autoref{prop:energyevolution} for the internal-energy term. Setting $\bm{U}_C = \bm{U}_L$ and using consistency gives:
    \begin{align*}
        0 &= p_{\rho}(\bm{\tau}_L)\left[f_h^\rho(\bm{U}_R,\bm{U}_L) - \rho_L V \right] + p_{\varepsilon}(\bm{\tau}_L)\left[f_h^\varepsilon(\bm{U}_R,\bm{U}_L;V) - \left(V(E_L+P) - \frac12 \rho_L V^3 \right) \right]  \nonumber \\
        &= p_{\rho}(\bm{\tau}_L)\left[f_h^\rho(\bm{U}_R,\bm{U}_L) - \rho_L V \right] + p_{\varepsilon}(\bm{\tau}_L)\left[f_h^\varepsilon(\bm{U}_R,\bm{U}_L;V) - \left(V(\varepsilon_L + \frac12 \rho_L V^2 +P) - \frac12 \rho_L V^3 \right) \right] \nonumber \\[0.5em]
        &= p_{\rho}(\bm{\tau}_L)\left[f_h^\rho(\bm{U}_R,\bm{U}_L) - \rho_L V \right] + p_{\varepsilon}(\bm{\tau}_L)\left[f_h^\varepsilon(\bm{U}_R,\bm{U}_L;V) - PV -\varepsilon_LV \right]. 
    \end{align*}
    In a similar manner, it follows from setting $\bm{U}_C = \bm{U}_R$ that:
    \begin{equation*}
        p_{\rho}(\bm{\tau}_R)\left[f_h^\rho(\bm{U}_R,\bm{U}_L) - \rho_R V \right] + p_{\varepsilon}(\bm{\tau}_R)\left[f_h^\varepsilon(\bm{U}_R,\bm{U}_L;V) - PV -\varepsilon_RV \right] = 0.
    \end{equation*}
    Since $V \neq 0$, we can define:
    \begin{equation*}
        \bm{\tau}_{LR} := \begin{bmatrix}
            \displaystyle \frac{f_h^\rho(\bm{U}_R,\bm{U}_L)}{V} \\[0.8em] \displaystyle\frac{f_h^\varepsilon(\bm{U}_R,\bm{U}_L;V)}{V} - P
        \end{bmatrix},
    \end{equation*}
    and the previous two equations imply $\bm{\tau}_{LR} - \bm{\tau}_L \in \ker(\left<\nabla p(\bm{\tau}_L),\cdot \right>)$ and $\bm{\tau}_{LR} - \bm{\tau}_R \in \ker(\left<\nabla p(\bm{\tau}_R),\cdot \right>)$. But then \autoref{lem:tangentspaceorthogonality} implies:
    \begin{equation*}
        \bm{\tau}_{LR} - \bm{\tau}_L \in T_{\bm{\tau}_L} \Gamma_P, \qquad \bm{\tau}_{LR} - \bm{\tau}_R \in T_{\bm{\tau}_R} \Gamma_P.
    \end{equation*}
    Adding $\bm{\tau}_L$ to $\bm{\tau}_{LR} - \bm{\tau}_L$ and $\bm{\tau}_R$ to $\bm{\tau}_{LR} - \bm{\tau}_R$ and using the definition of tangent lines then gives:
    \begin{equation*}
        \bm{\tau}_{LR}  \in L_{\bm{\tau}_L} \Gamma_P \cap L_{\bm{\tau}_R}\Gamma_P.
    \end{equation*}
    Thus $L_{\bm{\tau}_L} \Gamma_P \cap L_{\bm{\tau}_R}\Gamma_P \neq \varnothing$. Since $\bm{U}_L,\bm{U}_R \in \mathcal{M}_{P,V}$ were arbitrary this holds for any $\bm{\tau}_L,\bm{\tau}_R \in \Gamma_P$ by \autoref{prop:equilibriumset}, completing the proof.
\end{proof}

\begin{figure}
    \centering
    \includegraphics[width=0.6\linewidth]{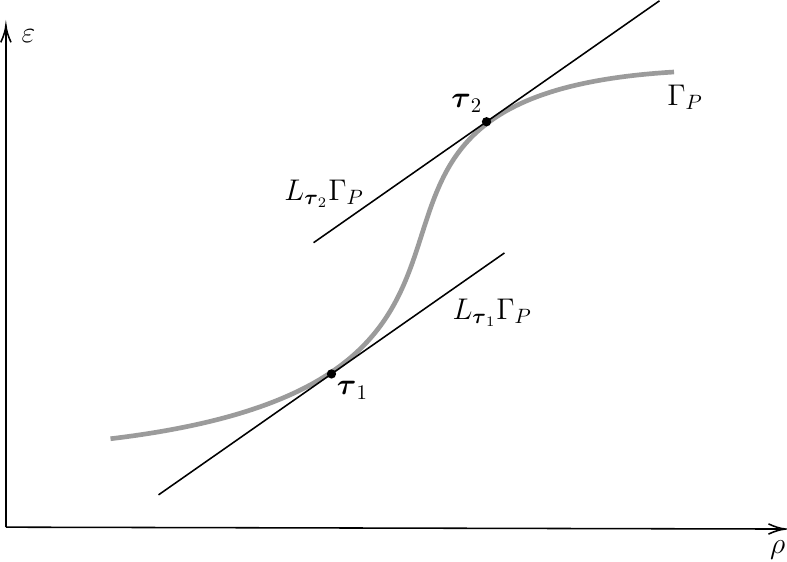}
    \caption{An impression of an isobar $\Gamma_P$ with points $\bm{\tau}_1,\bm{\tau}_2 \in \Gamma_P$ that have parallel but distinct tangent lines $L_{\bm{\tau}_1}\Gamma_P$ and $L_{\bm{\tau}_2}\Gamma_P$, respectively.}
    \label{fig:paraiso}
\end{figure}

\subsection{Sufficient conditions}
To finish our algebraic characterization we must show the converse of the implication in \autoref{lem:necessaryisobargeometry} also holds. That is, we must prove that if for any isobar $\Gamma_P$ the intersection between any two tangent lines is nonempty, then an algebraically PEP and consistent numerical-flux function exists. Thus, throughout this section we make the following assumption.
\begin{ass}[Isobar tangent-line intersection]
    \label{ass:tangentlineintersection}
    Let $\mathcal{S}_{\bm{\tau}}$ be an admissible supercritical regime.
    For every $P \in \mathbb{R}_+$ such that $\Gamma_P$ is nonempty it holds that:
    \begin{equation}
        L_{\bm{\tau}_1}\Gamma_P \cap L_{\bm{\tau}_2}\Gamma_P \neq \varnothing, \qquad \forall \bm{\tau}_1,\bm{\tau}_2 \in \Gamma_P.
    \end{equation}
\end{ass}
We will prove the existence of an algebraically PEP and consistent numerical-flux function based on \autoref{ass:tangentlineintersection} constructively. The numerical flux $\bm{f}_h : \mathcal{S}_{\bm{U}} \times \mathcal{S}_{\bm{U}} \rightarrow \mathbb{R}^3$ that we propose to demonstrate sufficiency is given as:
\begin{align}
    f_h^\rho(\bm{U}_R,\bm{U}_L) &:= \tilde{\rho}\, \overline{v}, \nonumber \\[0.3em] 
    f_h^m(\bm{U}_R,\bm{U}_L) &:= f_h^\rho\, \overline{v} + \overline{p}, \label{eq:flux} \\[0.3em] 
    f_h^E(\bm{U}_R,\bm{U}_L) &:= \overline{v}(\tilde{\varepsilon} + \tfrac12 f_h^\rho\, \overline{v} + \overline{p}), \nonumber
\end{align}
where $\tilde{\rho},\,\tilde{\varepsilon} : \mathcal{S}_{U}\times \mathcal{S}_{U} \rightarrow \mathbb{R}$ are special mass-density and internal-energy-density means, respectively, based on \autoref{ass:tangentlineintersection} and $\overline{v},\, \overline{p} : \mathcal{S}_{U}\times \mathcal{S}_{U} \rightarrow \mathbb{R}$ are means for the velocity and pressure, respectively. The particularly important design choice will be the means for mass and internal-energy density, however the numerical flux also requires some choice of mean for pressure and velocity. The main property that we will require of the pressure and velocity means is that $v(\bm{U}_L) = v(\bm{U}_R)= V$ implies $\overline{v}(\bm{U}_R,\bm{U}_L) = V$ and similarly for pressure, we will refer to this property as \emph{value consistency}. This property is satisfied by the arithmetic mean. Thus, for any real-valued function $a$ defined on $\mathcal{S}_{\bm{U}}$ or $\mathcal{S}_{\bm{\tau}}$, let $\overline{a}$ denote the arithmetic two-point mean i.e.\ $\overline{a}(\bm{U}_R,\bm{U}_L) := \tfrac12 (a(\bm{U}_L) + a(\bm{U}_R))$. 

Denoting $\bm{\tau}_R := \bm{\tau}(\bm{U}_R)$ and $\bm{\tau}_L := \bm{\tau}(\bm{U}_L)$, the value of the means $(\tilde{\rho},\,\tilde{\varepsilon})$ depends on whether $\bm{U}_R,\bm{U}_L \in \mathcal{M}_{P,V}$ holds or not for some $P \in \mathbb{R}_+$ and $V \in \mathbb{R}$. In case it does, the numerical flux must have the algebraic PEP property and the corresponding thermodynamic states $\bm{\tau}_L$ and $\bm{\tau}_R$ are on some shared isobar $\Gamma_P$. In case it does not, the flux is not constrained and needs only to be consistent. To satisfy the algebraic PEP property we will show that it is sufficient to pick the mean $(\tilde{\rho},\tilde{\varepsilon})$ between two points $\bm{\tau}_R,\bm{\tau}_L \in \Gamma_P$ in the intersection $L_{\bm{\tau}_L}\Gamma_P \cap L_{\bm{\tau}_R} \Gamma_P$ of the isobar tangent lines at these two points when $p(\bm{\tau}_R) = p(\bm{\tau}_L)$. For convenience, we will use this mean even if $v(\bm{U}_R) \neq v(\bm{U}_L)$. This intersection is never empty under \autoref{ass:tangentlineintersection} and, since tangent lines are affine subspaces of $\mathbb{R}^2$, consists of either a unique point or the entire tangent line as for example is the case for ideal gases. To deal with the possibly nonunique intersection we define the mean as follows:
\begin{equation}
    (\tilde{\rho},\,\tilde{\varepsilon}) := \tilde{\bm{\tau}}(\bm{U}_R,\bm{U}_L)=\begin{cases}
        \displaystyle\argmin_{\tilde{\bm{\tau}} \in \mathcal{I}} \Vert\tilde{\bm{\tau}} - \bm{\tau}^*\rVert, & p(\bm{\tau}_R) = p(\bm{\tau}_L)=P, \\[0.5em]
        \bm{\tau}^*, & p(\bm{\tau}_R) \neq p(\bm{\tau}_L),
    \end{cases} \qquad \mathcal{I} := L_{\bm{\tau}_L}\Gamma_P \cap L_{\bm{\tau}_R} \Gamma_P,
    \label{eq:massenergymean}
\end{equation}
where $\bm{\tau}^*$ denotes any consistent two-point mean without singularities and $\lVert\cdot\rVert$ is the Euclidean norm. If the tangent-line intersection is a single point, this point will be recovered from the minimization problem, otherwise the unique closest point to $\bm{\tau}^*$ will be selected. We note that a natural choice for the mean in case of $P_R := p(\bm{\tau}_R) \neq p(\bm{\tau}_L) =: P_L$ would have been to also pick a point in the intersection $(\tilde{\rho},\tilde{\varepsilon}) \in L_{\bm{\tau}_L}\Gamma_{P_L} \cap L_{\bm{\tau}_R} \Gamma_{P_R}$ which we can show corresponds to the choice of \cite{coppolapressure}. However, this would hinder our characterization as this requires this intersection to be nonempty, which is stronger than what we can prove must hold necessarily if a consistent, algebraically PEP numerical flux exists. Furthermore, since we only care about an existence characterization for algebraically PEP numerical-flux functions and not a fully well-defined numerical scheme, it is no problem to us that the mean in \eqref{eq:massenergymean} can be \emph{discontinuous}. As already mentioned in \autoref{ssec:algebraicpep}, extra regularity requirements on the flux would complicate the analysis and require the EOS to satisfy conditions that are not needed for PEP, but only for regularity of the flux.

Before addressing the algebraic PEP property we must address whether the flux 
\eqref{eq:flux} is consistent. Under \autoref{ass:tangentlineintersection}, this is indeed the case. 
\begin{prop}[Flux consistency]
    \label{prop:fluxconsistency}
    Under \autoref{ass:tangentlineintersection}, the flux $\bm{f}_h$ as in \eqref{eq:flux} with mass-density and internal-energy-density means $(\tilde{\rho},\tilde{\varepsilon})$ as in \eqref{eq:massenergymean} is well-defined and consistent in the sense of \autoref{dfn:consistency}.
\end{prop}
The proof is given in \autoref{ssec:fluxconsistency}.

We can now continue by proving that under \autoref{ass:tangentlineintersection}, the numerical flux \eqref{eq:flux} satisfies the velocity condition in \eqref{eq:apepconditions}. This can be done simply by using the velocity definition \eqref{eq:vpdefinition}. 
\begin{prop}[Velocity condition]
    \label{prop:velocityequilibrium}
    For any $P \in \mathbb{R}_+$ and $V \in \mathbb{R}$ the numerical flux $\bm{f}_h$ as in \eqref{eq:flux} satisfies:
    \begin{equation*}
        \nabla{v(\bm{U}_C)}^T\Delta \bm{f}_h(\bm{U}_R, \bm{U}_C, \bm{U}_L) = 0, \qquad \forall \bm{U}_R, \bm{U}_C, \bm{U}_L \in \mathcal{M}_{P,V}.
    \end{equation*}
\end{prop}
\begin{proof}
    Let $P \in \mathbb{R}_+$, $V \in \mathbb{R}$, $\bm{U}_R, \bm{U}_C, \bm{U}_L \in \mathcal{M}_{P,V}$ be arbitrary. Substituting the numerical-flux definitions \eqref{eq:flux} and \eqref{eq:massenergymean} and using the value-consistency property of the arithmetic velocity and pressure means:
    \begin{align*}
        &\nabla{v(\bm{U}_C)}^T\Delta \bm{f}_h(\bm{U}_R, \bm{U}_C, \bm{U}_L) \\
        &\qquad \qquad = -\frac{V}{\rho_C} \left(f_h^\rho(\bm{U}_R, \bm{U}_C) - f_h^\rho(\bm{U}_C, \bm{U}_L)\right) + \frac{1}{\rho_C} \left((f_h^\rho(\bm{U}_R, \bm{U}_C) V + P) -  (f_h^\rho(\bm{U}_C, \bm{U}_L) V + P)\right) \\
        &\qquad \qquad = 0.
    \end{align*}
\end{proof}

Finally, we can turn our attention to the pressure condition in \eqref{eq:apepconditions}. We will first analyze the $(\rho,\varepsilon)$-coordinates as an intermediate result. This will relate the gradient of the $(\rho,\varepsilon)$-coordinates in a central cell, $D\bm{\tau}(\bm{U}_C)$, and the difference of numerical fluxes on the cell's interfaces, $\Delta \bm{f}_h(\bm{U}_R,\bm{U}_C,\bm{U}_L)$, to the difference of the mass-density and internal-energy-density means, $\tilde{\bm{\tau}}(\bm{U}_R, \bm{U}_C) - \tilde{\bm{\tau}}(\bm{U}_C, \bm{U}_L)$. With our choice of $(\rho,\varepsilon)$-means this simple structure will lead to the algebraic PEP property as we will clarify in the proof of the pressure condition in \eqref{eq:apepconditions}.

\begin{prop}[Density and internal-energy conditions]
    \label{prop:densityenergyevolution}
    Under \autoref{ass:tangentlineintersection} and for any $P \in \mathbb{R}_+$, $V \in \mathbb{R}$ and $\bm{U}_R, \bm{U}_C, \bm{U}_L \in \mathcal{M}_{P,V}$, the numerical flux $\bm{f}_h$ as in \eqref{eq:flux} satisfies:
    \begin{equation}
        D\bm{\tau}(\bm{U}_C)\Delta \bm{f}_h(\bm{U}_R,\bm{U}_C,\bm{U}_L) = V(\tilde{\bm{\tau}}(\bm{U}_R, \bm{U}_C) - \tilde{\bm{\tau}}(\bm{U}_C, \bm{U}_L)).
        \label{eq:Tcevolution}
    \end{equation}
\end{prop}
\begin{proof}
    Let $P \in \mathbb{R}_+$, $V \in \mathbb{R}$, $\bm{U}_R, \bm{U}_C, \bm{U}_L \in \mathcal{M}_{P,V}$ be arbitrary. Denote $(\tilde{\rho}_{R,C}, \tilde{\varepsilon}_{R,C}) := \tilde{\bm{\tau}}(\bm{U}_R, \bm{U}_C)$ and similarly $(\tilde{\rho}_{C,L}, \tilde{\varepsilon}_{C,L}) := \tilde{\bm{\tau}}(\bm{U}_C, \bm{U}_L)$. Using the numerical flux \eqref{eq:flux}, the density satisfies:
    \begin{align*}
        \nabla \rho(\bm{U}_C)^T \Delta \bm{f}_h(\bm{U}_R,\bm{U}_C,\bm{U}_L)&= f_h^\rho(\bm{U}_R, \bm{U}_C) - f_h^\rho(\bm{U}_C, \bm{U}_L)  = \tilde{\rho}_{R,C}V - \tilde{\rho}_{C,L}V = V\left(\tilde{\rho}_{R,C} - \tilde{\rho}_{C,L}\right),
    \end{align*}
    by value consistency of the velocity mean. Using the definition of the internal-energy density in terms of conservative variables \eqref{eq:vpdefinition} and using the scheme \eqref{eq:discretization} with flux \eqref{eq:flux} gives:
    \begin{align*}
        &\nabla \varepsilon(\bm{U}_C)^T\Delta \bm{f}_h(\bm{U}_R,\bm{U}_C,\bm{U}_L)  \\
        &\qquad \qquad = \left(f_h^E(\bm{U}_R, \bm{U}_C) - f_h^E(\bm{U}_C, \bm{U}_L)  \right) + \frac12 V^2\left(f_h^\rho(\bm{U}_R, \bm{U}_C) - f_h^\rho(\bm{U}_C, \bm{U}_L)\right) \\
        &\qquad \qquad \qquad \qquad  - V\left(f_h^m(\bm{U}_R, \bm{U}_C) - f_h^m(\bm{U}_C, \bm{U}_L)\right) \\
        &\qquad \qquad = (V(\tilde{\varepsilon}_{R,C}+ \frac12 f_h^\rho(\bm{U}_R, \bm{U}_C) V + P)-V(\tilde{\varepsilon}_{C,L}+ \frac12f_h^\rho(\bm{U}_C, \bm{U}_L) V + P))  \\
        & \qquad \qquad \qquad \qquad + \frac12 V^2\left(f_h^\rho(\bm{U}_R, \bm{U}_C) - f_h^\rho(\bm{U}_C, \bm{U}_L)\right) - V((f_h^\rho(\bm{U}_R, \bm{U}_C) V + P) - (f_h^\rho(\bm{U}_C, \bm{U}_L)V + P) ) \\
        &\qquad \qquad = V\left(\tilde{\varepsilon}_{R,C} - \tilde{\varepsilon}_{C,L} \right).
    \end{align*}
    Gathering $\nabla \varepsilon(\bm{U}_C)^T\Delta \bm{f}_h(\bm{U}_R,\bm{U}_C,\bm{U}_L)$ and $\nabla \rho(\bm{U}_C)^T\Delta \bm{f}_h(\bm{U}_R,\bm{U}_C,\bm{U}_L)$ and using the definition of $\tilde{\bm{\tau}}_{R,C}$ and $\tilde{\bm{\tau}}_{C,L}$ gives the desired result.
\end{proof}

The reason that \eqref{eq:Tcevolution} with the mean \eqref{eq:massenergymean} leads to algebraic pressure equilibrium can be understood geometrically and is shown in \autoref{fig:dtaudt}. Suppose a classical solution of the semi-discrete equations \eqref{eq:discretization} exists, and consider an instant at which four neighboring conservative states $\bm{U}_1, \dots , \bm{U}_4 \in \mathcal{S}_{\bm{U}}$ belong to the same equilibrium set $\mathcal{M}_{P,V}$. Their thermodynamic coordinates lie on $\Gamma_P$. Combining the semi-discrete equations with the preceding algebraic identity gives, for $j = 2, 3$:
\begin{equation*}
    \dudt{\bm{\tau}_j} = -\frac1h D\bm{\tau}_j \Delta \bm{f}_h(\bm{U}_{j+1},\bm{U}_j,\bm{U}_{j-1}) = -\frac{V}{h}(\tilde{\bm{\tau}}_{j+1,j} - \tilde{\bm{\tau}}_{j,j-1}).
\end{equation*}
Semi-discretely, the pressure remains constant in cells $2$ and $3$ if $\ddtu{\bm{\tau}_2} \in T_{\bm{\tau}_2}\Gamma_P$ and $\ddtu{\bm{\tau}_3} \in T_{\bm{\tau}_3}\Gamma_P$ as then $\bm{\tau}_2$ and $\bm{\tau}_3$ are moving along an isobar. Since $\ddtu{\bm{\tau}_2}$ and $\ddtu{\bm{\tau}_3}$ have the above form, this can be arranged if the means $\tilde{\bm{\tau}}_{3,2}, \tilde{\bm{\tau}}_{2,1}$ lie on the tangent line $L_{\bm{\tau}_2}\Gamma_P$ and $\tilde{\bm{\tau}}_{4,3}, \tilde{\bm{\tau}}_{3,2}$ lie on the tangent line $L_{\bm{\tau}_3}\Gamma_P$. This is the case as the differences between these means are then tangent vectors to the isobar at $\bm{\tau}_2$ and $\bm{\tau}_3$, respectively. Now it becomes clear why the means $\tilde{\bm{\tau}}$ are chosen as points in the intersection of tangent lines. Namely, $\tilde{\bm{\tau}}_{3,2}$ then needs to lie on two tangent lines simultaneously, forcing it to lie in the intersection $L_{\bm{\tau}_2}\Gamma_P \cap L_{\bm{\tau}_3}\Gamma_P$, which is precisely how we constructed it. The following lemma uses exactly this reasoning. This interpretation is conditional on solution existence, however the following proof establishes the algebraic pressure identities directly.

\begin{figure}
    \centering
    \includegraphics[width=0.6\linewidth]{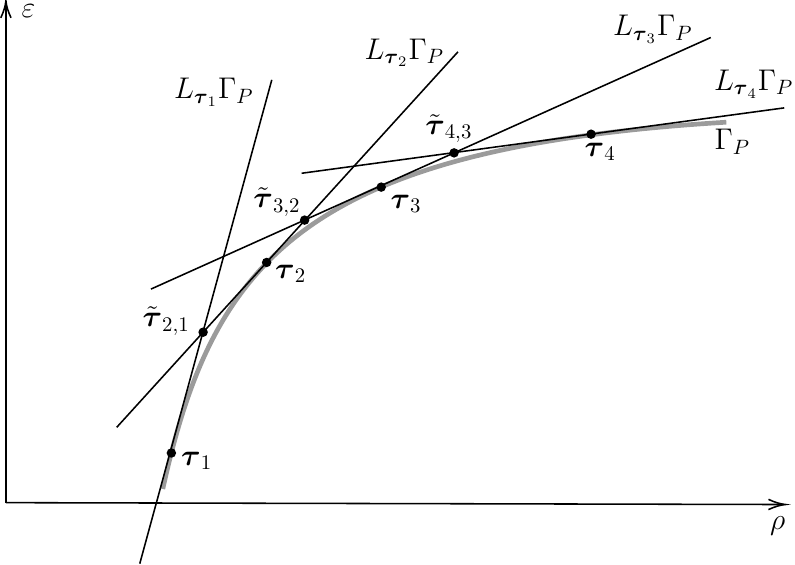}
    \caption{An impression of an isobar $\Gamma_P$ in $(\rho,\varepsilon)$-coordinates with four grid values on it and the choice of mass-density and internal-energy-density means between them.}
    \label{fig:dtaudt}
\end{figure}

\begin{lem}[Isobar tangent-line sufficiency]
    \label{lem:isobartangentlinesufficiency}
    Under \autoref{ass:tangentlineintersection}, the numerical flux $\bm{f}_h$ as in \eqref{eq:flux} is well-defined, consistent and algebraically PEP. 
\end{lem}
\begin{proof}
    \autoref{prop:fluxconsistency} shows the numerical flux $\bm{f}_h$ is consistent, while \autoref{prop:velocityequilibrium} shows the numerical flux satisfies the velocity condition in \eqref{eq:apepconditions}. It remains to show that the pressure condition in \eqref{eq:apepconditions} holds. Let $P \in \mathbb{R}_+$, $V \in \mathbb{R}$ and $\bm{U}_R,\bm{U}_C,\bm{U}_L\in\mathcal{M}_{P,V}$ be arbitrary. Denote $\bm{\tau}_R := \bm{\tau}(\bm{U}_R)$, $\bm{\tau}_C := \bm{\tau}(\bm{U}_C)$ and $\bm{\tau}_L := \bm{\tau}(\bm{U}_L)$. \autoref{prop:densityenergyevolution} shows that $\bm{\tau}$ and the numerical flux satisfy:
    \begin{equation*}
        D\bm{\tau}(\bm{U}_C)\Delta \bm{f}_h(\bm{U}_R,\bm{U}_C,\bm{U}_L) = V(\tilde{\bm{\tau}}(\bm{U}_R, \bm{U}_C) - \tilde{\bm{\tau}}(\bm{U}_C, \bm{U}_L)).
    \end{equation*}
    By construction it holds that:
    \begin{equation*}
        \tilde{\bm{\tau}}(\bm{U}_R, \bm{U}_C) \in L_{\bm{\tau}_R}\Gamma_P \cap L_{\bm{\tau}_C}\Gamma_P \subset L_{\bm{\tau}_C}\Gamma_P, \qquad \tilde{\bm{\tau}}(\bm{U}_C, \bm{U}_L) \in L_{\bm{\tau}_C}\Gamma_P \cap L_{\bm{\tau}_L}\Gamma_P \subset L_{\bm{\tau}_C}\Gamma_P.
    \end{equation*}
    Thus, by definition of tangent lines, $\tilde{\bm{\tau}}(\bm{U}_R, \bm{U}_C) = \bm{\tau}_C + \bm{v}_R$ for some $\bm{v}_R \in T_{\bm{\tau}_C}\Gamma_P$ and similarly $\tilde{\bm{\tau}}(\bm{U}_C, \bm{U}_L) = \bm{\tau}_C + \bm{v}_L$ for some $\bm{v}_L \in T_{\bm{\tau}_C}\Gamma_P$. Subtracting the means then gives:
    \begin{equation*}
        D\bm{\tau}(\bm{U}_C)\Delta \bm{f}_h(\bm{U}_R,\bm{U}_C,\bm{U}_L) = V(\tilde{\bm{\tau}}(\bm{U}_R, \bm{U}_C) - \tilde{\bm{\tau}}(\bm{U}_C, \bm{U}_L)) = V(\bm{v}_R - \bm{v}_L) \in T_{\bm{\tau}_C}\Gamma_P.
    \end{equation*}
    However, by \autoref{lem:tangentspaceorthogonality} we have $T_{\bm{\tau}_C}\Gamma_P = \ker(\left<\nabla p(\bm{\tau}_C), \cdot \right>)$. This implies:
    \begin{align*}
        \nabla p(\bm{\tau}_C)^TD\bm{\tau}(\bm{U}_C)\Delta \bm{f}_h(\bm{U}_R,\bm{U}_C,\bm{U}_L) &= \left< \nabla p(\bm{\tau}_C), D\bm{\tau}(\bm{U}_C)\Delta \bm{f}_h(\bm{U}_R,\bm{U}_C,\bm{U}_L)\right> = 0,
    \end{align*}
    thus the numerical flux satisfies both identities in \autoref{dfn:algebraicpressureequilibriumpreservation}.
\end{proof}

\subsection{Existence characterization}
We are now ready to formally state our main result: a full characterization of the existence of consistent and algebraically PEP numerical-flux functions given any EOS possessing an admissible supercritical regime as defined in \autoref{dfn:admissiblesupercriticalregime}.

\begin{thm}[Algebraic-PEP-flux existence]\label{thm:pepexistence}
    Fix an admissible supercritical regime $\mathcal{S}_{\bm{\tau}}$ and set $\mathcal{S}_{\bm{U}}=\bm{\Theta}(\mathcal{S}_{\bm{\tau}}\times\mathbb R)$. A numerical flux $\bm{f}_h : \mathcal{S}_{\bm{U}} \times \mathcal{S}_{\bm{U}} \rightarrow \mathbb{R}^3$ that is \emph{consistent} in the sense of \autoref{dfn:consistency} and is \emph{algebraically pressure-equilibrium-preserving} in the sense of \autoref{dfn:algebraicpressureequilibriumpreservation} exists if and only if for any $P \in \mathbb{R}_+$ so that the isobar $\Gamma_P$ is nonempty any two tangent lines have a nonempty intersection, that is:
    \begin{equation*}
        L_{\bm{\tau}_1}\Gamma_P \cap L_{\bm{\tau}_2}\Gamma_P \neq \varnothing, \qquad \forall \bm{\tau}_1,\bm{\tau}_2 \in \Gamma_P.
    \end{equation*}
\end{thm}
\begin{proof}
    Necessity is shown by \autoref{lem:necessaryisobargeometry}, while sufficiency is shown by \autoref{lem:isobartangentlinesufficiency}.
\end{proof}

Here, $\bm{\Theta}$ is the diffeomorphism in \eqref{eq:conservativediffeomorphism}. This characterization shows that the existence of an algebraically PEP flux is a geometric property of $(\rho,\varepsilon)$-isobars of the EOS on an admissible supercritical regime. The existence of such a flux requires that the tangent lines at any two points on an isobar have nonempty intersections. Note that this should even hold for tangent lines at points in different connected components of an isobar. This geometric condition can be checked visually by comparing tangent lines of pairs of points. This a nonlocal check. Under some reasonable additional assumptions on the isobar, local conditions similar to the inflection as shown in \autoref{fig:paraiso} become necessary and sufficient for the existence of disjoint tangent lines. We discuss this in more detail in \autoref{ssec:curvaturethm}.

We have characterized numerical-flux functions that are PEP only in the algebraic sense of \autoref{dfn:algebraicpressureequilibriumpreservation}. This is because in our sufficiency argument we have constructed a numerical flux with a possibly discontinuous mean \eqref{eq:massenergymean} for the mass density and internal-energy density. A semi-discretization \eqref{eq:discretization} using this discontinuous mean in its flux may not be sufficiently regular to have a solution $\bm{U}_i(t)$. Thus our geometric condition in \autoref{thm:pepexistence} may not be sufficient for the existence of a scheme \eqref{eq:discretization} that is PEP in the sense of \autoref{dfn:pressureequilibriumpreservation}. However, the numerical flux of a PEP scheme is certainly algebraically PEP and thus the geometric condition is necessary for the existence of a PEP scheme, proving the following corollary.
\begin{cor}[PEP-scheme criterion]\label{cor:pepschemeexistence}
    Fix an admissible supercritical regime $\mathcal{S}_{\bm{\tau}}$ and set $\mathcal{S}_{\bm{U}}=\bm{\Theta}(\mathcal{S}_{\bm{\tau}} \times\mathbb R)$. If a conservative scheme of the form \eqref{eq:discretization} with a consistent flux and the PEP property as in \autoref{dfn:pressureequilibriumpreservation} exists then any pair of points on a nonempty isobar has intersecting tangent lines.
\end{cor}
In other words, no PEP schemes of the considered form can exist on a domain if it contains two states on the same isobar whose tangent lines do not intersect.

Since tangent lines are one-dimensional affine subspaces of $\mathbb{R}^2$ there are three cases for the intersection of any two tangent lines to the same isobar. We list these possibilities and their implication for the existence of algebraic PEP fluxes in \autoref{tab:tangentlineintersections}.

\begin{table}[ht]
    \centering
    \begin{tabular}{p{0.18\textwidth} p{0.20\textwidth} p{0.25\textwidth} p{0.22\textwidth}}
        \hline
        \textbf{Case} & \textbf{Tangent-line geometry} & \textbf{Intersection} & \textbf{Implication for PEP existence} \\
        \hline
        Unique intersection
        &
        The tangent lines are not parallel,
        \[
            L_{\bm{\tau}_1}\Gamma_P \nparallel L_{\bm{\tau}_2}\Gamma_P .
        \]
        &
        \[
            L_{\bm{\tau}_1}\Gamma_P
            \cap
            L_{\bm{\tau}_2}\Gamma_P
            =
            \{\widetilde{\bm{\tau}}\}.
        \]
        &
        The condition of \autoref{thm:pepexistence} is satisfied for this pair. The compatible intersection point is unique.
        \\[1ex]
        \hline

        Coincident tangent lines
        &
        The two tangent lines coincide,
        \[
            L_{\bm{\tau}_1}\Gamma_P
            =
            L_{\bm{\tau}_2}\Gamma_P .
        \]
        &
        \[
        \begin{aligned}
            L_{\bm{\tau}_1}\Gamma_P \cap L_{\bm{\tau}_2}\Gamma_P &= L_{\bm{\tau}_1}\Gamma_P \\
            &= L_{\bm{\tau}_2}\Gamma_P.
        \end{aligned}
        \]
        &
        The condition of \autoref{thm:pepexistence} is satisfied for this pair. The compatible intersection point is nonunique.
        \\[1ex]
        \hline

        Parallel distinct tangent lines
        &
        The tangent lines are parallel but do not coincide,
        \[
        \begin{gathered}
            L_{\bm{\tau}_1}\Gamma_P
            \parallel
            L_{\bm{\tau}_2}\Gamma_P \, \text{ and} \\
            L_{\bm{\tau}_1}\Gamma_P
            \neq
            L_{\bm{\tau}_2}\Gamma_P .
        \end{gathered}
        \]
        &
        \[
            L_{\bm{\tau}_1}\Gamma_P
            \cap
            L_{\bm{\tau}_2}\Gamma_P
            =
            \varnothing .
        \]
        &
        The condition of \autoref{thm:pepexistence} is violated. Hence, no consistent and algebraically PEP flux exists on an admissible supercritical regime containing this pair and no PEP schemes exist according to \autoref{cor:pepschemeexistence}.
        \\
        \hline
    \end{tabular}
    \caption{Possible intersections of two tangent lines to the same isobar.}
    \label{tab:tangentlineintersections}
\end{table}

\begin{table}[ht]
    \centering
    \begin{tabular}{p{0.15\textwidth} p{0.24\textwidth} p{0.27\textwidth} p{0.22\textwidth}}
        \hline
        \textbf{Case}
        &
        \textbf{Geometric condition}
        &
        \textbf{Algebraic condition}
        &
        \textbf{Consequence}
        \\
        \hline

        Unique intersection
        &
        The tangent lines are nonparallel:
        \[
            L_{\bm{\tau}_1}\Gamma_P \nparallel L_{\bm{\tau}_2}\Gamma_P .
        \]
        &
        Different tangent-line slopes:\par
        \centering
        \(\displaystyle
            \Delta\varepsilon_{\rho}\neq 0.
        \)
        \par
        \raggedright
        Equation \eqref{eq:chan1} uniquely determines:\par
        \centering
        \(\displaystyle
            \tilde{\rho}
            =
            \frac{
                \Delta(\rho\varepsilon_{\rho})
                -
                \Delta\varepsilon
            }{
                \Delta\varepsilon_{\rho}
            }.
        \)
        \par
        \raggedright
        Then \eqref{eq:chan2} uniquely determines
        $\tilde{\varepsilon}$.
        &
        There exists a unique compatible pair:
        \[
            (\tilde{\rho},\tilde{\varepsilon})
            \in
            L_{\bm{\tau}_1}\Gamma_P
            \cap
            L_{\bm{\tau}_2}\Gamma_P .
        \]
        \\[1.5ex]
        \hline

        Coincident tangent lines
        &
        The tangent lines have equal slope and coincide:
        \[
            L_{\bm{\tau}_1}\Gamma_P
            =
            L_{\bm{\tau}_2}\Gamma_P .
        \]
        &
        Same tangent-line slope and same $(\rho=0)$-intercept:\par
        \centering
        \(\displaystyle
        \begin{gathered}
            \Delta\varepsilon_{\rho}=0,
            \\
            \Delta(\rho\varepsilon_{\rho})
            -
            \Delta\varepsilon
            =0,
        \end{gathered}
        \)
        \par
        \raggedright
        so \eqref{eq:chan1} reduces to:\par
        \centering
        \(\displaystyle
            0=0.
        \)
        \par
        \raggedright
        Equation \eqref{eq:chan2} becomes:\par
        \centering
        \(\displaystyle
            \tilde{\varepsilon}
            =
            \overline{\varepsilon}
            +
            \varepsilon_{\rho}\left(
                \tilde{\rho}-\overline{\rho}
            \right).
        \)
        \par
        \raggedright
        &
        There are infinitely many compatible pairs. They are precisely the points on the common tangent line:
        \[
            (\tilde{\rho},\tilde{\varepsilon})
            \in
            L_{\bm{\tau}_1}\Gamma_P
            =
            L_{\bm{\tau}_2}\Gamma_P .
        \]
        \\[1.5ex]
        \hline

        Parallel distinct tangent lines
        &
        The tangent lines are parallel but do not coincide:
        \[
            \begin{gathered}
            L_{\bm{\tau}_1}\Gamma_P
            \parallel
            L_{\bm{\tau}_2}\Gamma_P \, \text{ and} \\
            L_{\bm{\tau}_1}\Gamma_P
            \neq
            L_{\bm{\tau}_2}\Gamma_P .
        \end{gathered}
        \]
        &
        Same tangent-line slope, but different intercept at $\rho=0$, thus:\par
        \centering
        \(\displaystyle
            \Delta\varepsilon_{\rho}=0,
        \)
        \par
        \raggedright
        while:\par
        \centering
        \(\displaystyle
            \Delta(\rho\varepsilon_{\rho})
            -
            \Delta\varepsilon
            \neq 0.
        \)
        \par
        \raggedright
        Hence \eqref{eq:chan1} cannot be satisfied.
        &
        No compatible pair
        $(\tilde{\rho},\tilde{\varepsilon})$
        exists. Equivalently:
        \[
            L_{\bm{\tau}_1}\Gamma_P
            \cap
            L_{\bm{\tau}_2}\Gamma_P
            =
            \varnothing .
        \]        
        \\
        \hline
    \end{tabular}
    \caption{Geometric and algebraic characterization of the possible intersections of tangent lines at two distinct states $\bm{\tau}_1,\bm{\tau}_2 \in \Gamma_P$, assuming that $\varepsilon_{\rho}$ is defined at both states.}
    \label{tab:chanintersectioncases}
\end{table}

\subsubsection{Comparison with Chan et al. \cite{channodal}}\label{ssec:comparisonchan}
Recently Chan et al. \cite{channodal} derived an algebraic compatibility condition for the density and internal-energy-density means $(\tilde{\rho},\tilde{\varepsilon})$ of the numerical flux of a formally PEP numerical scheme for arbitrary EOS. This condition was obtained by algebraically characterizing the equivalent three-point PEP compatibility condition in \cite{terashimaapproximately} using a technique similar to that used to characterize the entropy-conservation condition in \cite{artianoaffordable}. It is formulated in terms of the derivative $\varepsilon_{\rho}(\rho) = \left.(\partial \varepsilon/\partial \rho)\right|_p$ of the internal-energy density $\varepsilon$ with respect to density $\rho$ at some constant pressure $P \in \mathbb{R}_+$. We now show that, whenever this derivative exists, their compatibility condition has a direct geometric interpretation. Namely, a mean $(\tilde{\rho},\tilde{\varepsilon})$ satisfying their compatibility condition is precisely a point belonging to the intersection of the tangent lines of the corresponding isobar. Consequently, the compatibility condition of \cite{channodal} is a coordinate representation of the more general geometric condition appearing in \autoref{thm:pepexistence}. 

The derivative $\varepsilon_{\rho}(\rho)$ can be defined in terms of pressure derivatives as follows:
\begin{equation}
    \varepsilon_{\rho}(\rho) = -\frac{p_{\rho}(\rho,\varepsilon(\rho))}{p_{\varepsilon}(\rho,\varepsilon(\rho))}, \qquad P = p(\rho, \varepsilon(\rho)),
    \label{eq:energyderivative}
\end{equation}
where $\varepsilon(\rho)$ is a graph of the isobar $\Gamma_P$ in $(\rho,\varepsilon)$-space. Conditions for this graph are given in \autoref{thm:curvature} in \autoref{ssec:curvaturethm}. The expression \eqref{eq:energyderivative} follows from taking the derivative with respect to $\rho$ of the right equation and solving for $\varepsilon_{\rho}(\rho)$. Now, for any function $a$ on $\mathcal{S}_{\bm{\tau}}$, denote $\Delta a(\bm{\tau}_1, \bm{\tau}_2) := a(\bm{\tau}_2) - a(\bm{\tau}_1)$ for any $\bm{\tau}_1,\bm{\tau}_2 \in \mathcal{S}_{\bm{\tau}}$ and similarly for functions on $\mathcal{S}_{\bm{U}}$. The PEP compatibility condition of \cite{channodal} is then given as:
\begin{align}
    \tilde{\rho} \Delta \varepsilon_{\rho}&= \Delta (\rho \varepsilon_{\rho}) - \Delta \varepsilon, \label{eq:chan1}\\
    \tilde{\varepsilon}&= \overline{\varepsilon} + \tilde{\rho}\,\overline{\varepsilon_{\rho}} - \overline{\rho \varepsilon_{\rho}}, \label{eq:chan2}
\end{align}
for all $\bm{\tau}_1,\bm{\tau}_2 \in \Gamma_P$, where $\overline{a}$ is the arithmetic mean. 

Assume that $\rho$ is indeed a valid coordinate to parameterize an isobar and thus that $\varepsilon(\rho)$ is the graph of an isobar. Then $\varepsilon_{\rho}(\rho)$ is the slope of the tangent line at $(\rho,\varepsilon(\rho))$ for this isobar. Moreover, if $\bm{\tau}_1,\bm{\tau}_2 \in \Gamma_P$, tangent lines at these points can be defined as:
\begin{align*}
    L_{\bm{\tau}_1}\Gamma_P &= \{ (\rho, \varepsilon) \in \mathbb{R}^2 \, :\, \varepsilon = \varepsilon_{\rho}(\rho_1)(\rho - \rho_1) + \varepsilon_1  \}, \\
    L_{\bm{\tau}_2}\Gamma_P &= \{ (\rho, \varepsilon) \in \mathbb{R}^2 \, :\, \varepsilon = \varepsilon_{\rho}(\rho_2)(\rho - \rho_2) + \varepsilon_2  \},
\end{align*}
where $(\rho_1,\varepsilon_1) := \bm{\tau}_1$ and $(\rho_2,\varepsilon_2) := \bm{\tau}_2$. A point $(\rho_{\mathrm{int}},\varepsilon_{\mathrm{int}}) \in \mathbb{R}^2$ is in the intersection of the tangent lines if and only if:
\begin{equation*}
    \varepsilon_{\mathrm{int}} = \varepsilon_{\rho}(\rho_1)(\rho_{\mathrm{int}} - \rho_1) + \varepsilon_1 = \varepsilon_{\rho}(\rho_2)(\rho_{\mathrm{int}} - \rho_2) + \varepsilon_2.
\end{equation*}
The second equation can be rearranged to obtain \eqref{eq:chan1} evaluated between $\bm{\tau}_1$ and $\bm{\tau}_2$:
\begin{equation*}
    \rho_{\mathrm{int}}\Delta \varepsilon_{\rho}(\bm{\tau}_2,\bm{\tau}_1) = \Delta(\rho \varepsilon_{\rho})(\bm{\tau}_2,\bm{\tau}_1) - \Delta \varepsilon(\bm{\tau}_2,\bm{\tau}_1).
\end{equation*}
Furthermore, we can compute the corresponding $\varepsilon$-coordinate at this intersection by evaluating either the parameterization of $L_{\bm{\tau}_1}\Gamma_P$ or $L_{\bm{\tau}_2}\Gamma_P$:
\begin{align*}
    \varepsilon_{\mathrm{int}} &= \varepsilon_{\rho}(\rho_1)(\rho_{\mathrm{int}} - \rho_1) + \varepsilon_1, \\
    \varepsilon_{\mathrm{int}} &= \varepsilon_{\rho}(\rho_2)(\rho_{\mathrm{int}} - \rho_2) + \varepsilon_2,
\end{align*}
and \eqref{eq:chan2} evaluated between $\bm{\tau}_1$ and $\bm{\tau}_2$ can then be obtained by averaging:
\begin{align*}
    \varepsilon_{\mathrm{int}} &= \frac{\varepsilon_{\mathrm{int}} + \varepsilon_{\mathrm{int}}}{2} \\
    &= \tfrac12 (\varepsilon_{\rho}(\rho_1)(\rho_{\mathrm{int}} - \rho_1) + \varepsilon_1) + \tfrac12 (\varepsilon_{\rho}(\rho_2)(\rho_{\mathrm{int}} - \rho_2) + \varepsilon_2) \\[0.7em]
    &= \overline{\varepsilon}(\bm{\tau}_2,\bm{\tau}_1) + \rho_{\mathrm{int}}\overline{\varepsilon_{\rho}}(\bm{\tau}_2,\bm{\tau}_1) - \overline{\rho \varepsilon_{\rho}}(\bm{\tau}_2,\bm{\tau}_1).
\end{align*}
Thus provided $\varepsilon_{\rho}$ is defined at $\bm{\tau}_1,\bm{\tau}_2 \in \Gamma_P$, a pair $(\tilde{\rho},\tilde{\varepsilon})$ satisfies the compatibility condition of Chan et al. \cite{channodal} if and only if it belongs to $L_{\bm{\tau}_1}\Gamma_P \cap L_{\bm{\tau}_2}\Gamma_P$. In turn, if $\varepsilon_{\rho}$ is defined, this implies that the solvability of the compatibility conditions by Chan et al. \cite{channodal} for all pairs of points on an isobar $\Gamma_P$ is equivalent to the nonemptiness of tangent-line intersections $L_{\bm{\tau}_1}\Gamma_P \cap L_{\bm{\tau}_2}\Gamma_P$ for all these pairs i.e.\ the geometric condition of \autoref{thm:pepexistence}. The geometric cases are listed in \autoref{tab:tangentlineintersections}. When $\varepsilon_{\rho}$ exists, the equivalence above gives the corresponding algebraic conditions in \autoref{tab:chanintersectioncases}.

These equivalences rely on the existence of $\varepsilon_{\rho}(\rho)$. However, from basic thermodynamic principles alone this does not follow necessarily. Namely, for the term $p_{\varepsilon}$ in \eqref{eq:energyderivative} we can derive:
\begin{equation*}
    p_{\varepsilon}(\bm{\Phi}(\bm{\eta})) = - \frac{e_{\nu\sigma}(\bm{\eta}) \nu}{T(\bm{\eta})}, 
\end{equation*}
for $\bm{\eta} := (\nu,\sigma) \in \mathcal{S}_{\bm{\eta}}$ where $\bm{\Phi}$ is defined in \eqref{eq:Phi}. While for all $\bm{\eta}\in\mathcal{S}_{\bm{\eta}}$ it holds that $\nu > 0$ by \eqref{eq:thermodynamicstatespace} and $T(\bm{\eta}) > 0$ by the monotonicity condition in \autoref{dfn:admissiblesupercriticalregime}, the sign of $e_{\nu\sigma}$ is not constrained on an admissible supercritical regime $\mathcal{S}_{\bm{\eta}}$. Hence, $p_{\varepsilon}$ can potentially vanish so that the expression $\varepsilon_{\rho}(\rho)$ is not defined. Nonetheless, for many fluids of engineering interest this is not the case \cite{mausbachcomparative}. Counterexamples do exist, one such fluid is supercooled water \cite{mausbachcomparative}. Our geometric approach does not require $\rho$ to be a valid isobar coordinate, which is why we use it throughout the analysis.

\subsubsection{Accuracy issues of PEP schemes}\label{ssec:accuracyissues}
Assuming $\varepsilon_{\rho}$ is defined and $V \neq 0$, \autoref{tab:chanintersectioncases} shows that on a $(P,V)$-equilibrium any consistent and algebraically PEP numerical flux evaluated between points with nonparallel isobar tangent lines forces a unique pair of mass-density and internal-energy-density means. Indeed, it is shown in \cite{channodal} that, in this case, their numerical flux reduces to that of \cite{coppolapressure} and the single-component version of \cite{degrendeleconstruction}. Moreover, at nonzero common velocity, the uniqueness of these means for consistent and algebraically PEP fluxes allows us to make some general comments on the treatment of singularities and its effect on the accuracy of these means. One possible treatment of singularities proposed in \cite{channodal} based on suggestions in \cite{coppolapressure} is to switch to a nonsingular numerical flux that need not satisfy algebraic PEP when:
\begin{equation}
    |\varepsilon_{\rho}(\rho_R) - \varepsilon_{\rho}(\rho_L)| \leq C \epsilon_{\mathrm{mach}},
    \label{eq:absoluteswitch}
\end{equation}
where $\bm{\tau}_R = (\rho_R,\varepsilon_R)$ and $\bm{\tau}_L = (\rho_L,\varepsilon_L)$ are the thermodynamic states corresponding to the right and left arguments of a numerical flux, $\epsilon_{\mathrm{mach}} \in \mathbb{R}_+$ is the machine precision and $C \in \mathbb{R}_+$ is a user-specified constant, for example $C = 100$ in \cite{channodal}. This switch can be slightly problematic because it depends on the choice of units. Thus, expressing $\varepsilon_{\rho}$ in $\mathrm{kJ}\, \mathrm{kg}^{-1}$ has different switching behavior compared to using $\mathrm{MJ}\, \mathrm{kg}^{-1}$. A reasonable alternative is a switch based on the relative slope difference:
\begin{equation}
    |\varepsilon_{\rho}(\rho_R) - \varepsilon_{\rho}(\rho_L)| \leq C \epsilon_{\mathrm{mach}}\max(|\varepsilon_{\rho}(\rho_R) |,| \varepsilon_{\rho}(\rho_L)|).
    \label{eq:singularityswitch}
\end{equation} 
Since $\varepsilon_{\rho}$, when it is defined, is the slope of the tangent line, both conditions \eqref{eq:singularityswitch} and \eqref{eq:absoluteswitch} prevent evaluating the PEP numerical flux between points $\bm{\tau}_L, \bm{\tau}_R$ at which the tangent lines are parallel. Hence, the conditions do indeed prevent singular pairs from being evaluated as for these pairs \autoref{tab:chanintersectioncases} shows that their tangent lines are parallel and noncoincident. However, since usually $C \epsilon_{\mathrm{mach}} \ll 1$, states with nearly parallel tangent lines are still allowed by either switch. In this case, their intersection $\tilde{\bm{\tau}}$ could lie far from both states $\bm{\tau}_L,\bm{\tau}_R$. At nonzero common velocity, the unique tangent-line intersection determines the mass-density and internal-energy density means of every consistent, algebraically PEP flux. Consequently, these means can lie far outside the range of the input states or even be thermodynamically inadmissible, for example when $\tilde{\rho}<0$. We will demonstrate this poor accuracy in more detail in the next section.

\begin{rmk}
    Both switches \eqref{eq:singularityswitch} and \eqref{eq:absoluteswitch} also activate when tangent lines are parallel and coincident, e.g. for the ideal-gas law $p = (\gamma-1)\varepsilon$. However, in these situations \autoref{tab:chanintersectioncases} shows that infinitely many compatible means exist. Whether switching preserves algebraic PEP then depends on the replacement flux.
\end{rmk}

\section{Experiments}\label{sec:experiments}
We will now put our existence characterization into practice. Even though we have only characterized the algebraic PEP property of \autoref{dfn:algebraicpressureequilibriumpreservation}, we will still make statements about the existence of PEP schemes \eqref{eq:discretization} satisfying \autoref{dfn:pressureequilibriumpreservation} by using \autoref{cor:pepschemeexistence}. Our experiments will show that for both simple and more complex real-gas EOS for supercritical fluids, consistent, conservative and PEP schemes \emph{cannot be expected to exist} without singularities for reasonable definitions of the supercritical regime. Additionally, PEP schemes for supercritical fluids are usually claimed to be beneficial especially across supercritical gas-like/liquid-like material interfaces where they prevent pressure oscillations. We will show that, in fact, current conservative PEP schemes using switches of the form \eqref{eq:singularityswitch} can also cause quite severe stability issues across these interfaces regardless of whether these interfaces are resolved or under-resolved. Thus \emph{enforcing the PEP property can come at the price of accuracy and stability}. Our novel geometric perspective will make the mechanism behind these issues clear.

In particular, we will carry out three experiments for two molecules using different EOS that are of practical relevance. Specifically, we will investigate supercritical carbon dioxide ($\mathrm{CO}_2$) and supercritical nitrogen ($\mathrm{N}_2$). For carbon dioxide we use the complex, high-accuracy Span--Wagner reference EOS \cite{SpanWagner1996}, whereas for nitrogen we use the simpler Peng--Robinson EOS \cite{PengRobinson1976,BellJager2016}. In our experiments all thermodynamic states, for both Span--Wagner using the HEOS implementation and Peng--Robinson, are evaluated using CoolProp 8.0.0 \cite{Bell2014CoolProp}. Hence, we refer to \cite{Bell2014CoolProp} for more information on the specific data used in our experiments. The code to reproduce all experiments is made available at \cite{reprorepo}. In our first experiment, we will give a broad impression of the existence of PEP schemes for the stated EOS-molecule pairs. In our second experiment we will do a more detailed numerical search for singular pairs along a supercritical isobar passing close to the critical point where thermodynamic behavior is usually most extreme. In our third experiment we will demonstrate how the geometry of our chosen real-gas EOS together with the PEP property can require the use of poor thermodynamic means resulting in quite severe numerical stability issues for both resolved and under-resolved flows.

\begin{figure}[!p]
    \centering

    \includegraphics[width=0.8\linewidth]{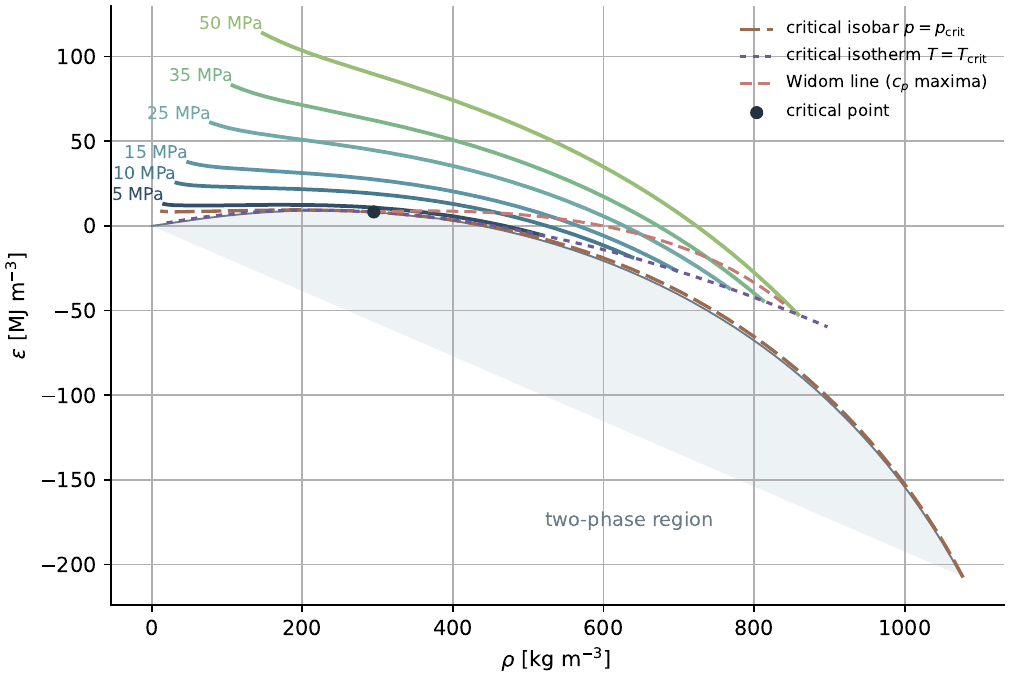}
    \caption{Supercritical nitrogen isobars in $(\rho,\varepsilon)$-coordinates
    computed using the Peng--Robinson EOS. The critical isobar and isotherm,
    critical point, two-phase region, and Widom line are shown for reference.}
    \label{fig:isobarsn2}

    \vspace{1em}

    \includegraphics[width=0.8\linewidth]{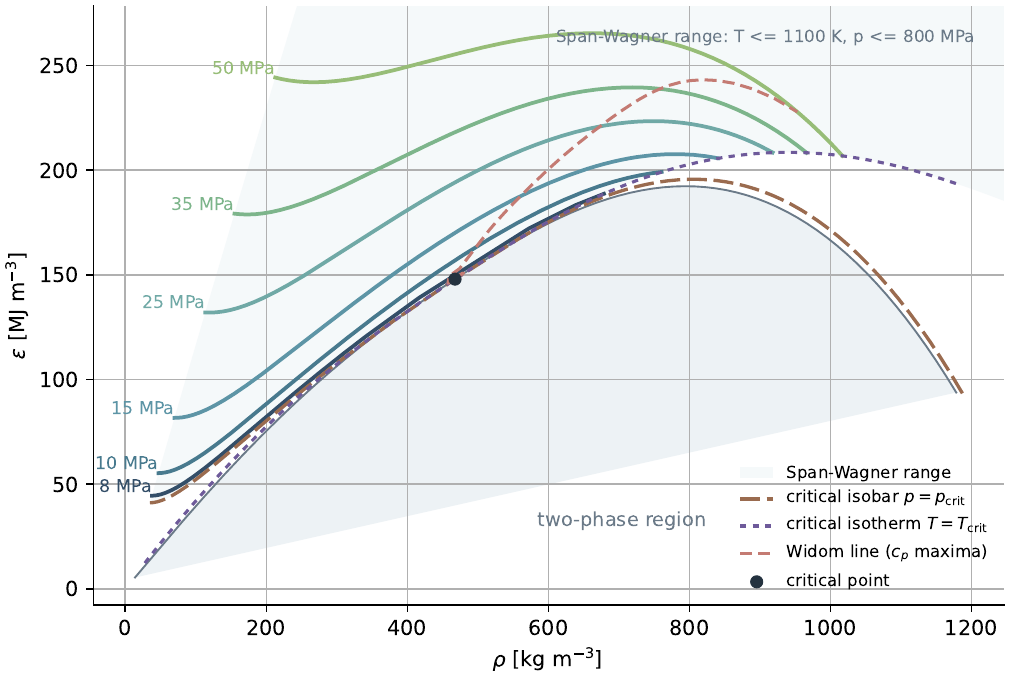}
    \caption{Supercritical carbon-dioxide isobars in $(\rho,\varepsilon)$-coordinates
    computed using the Span--Wagner EOS. The critical isobar and isotherm,
    critical point, two-phase region, Widom line, and admissible Span--Wagner
    range are shown for reference.}
    \label{fig:isobarsco2}

\end{figure}

\subsection{General isobar geometry}
In this experiment we provide an initial broad impression on the existence of PEP schemes for realistic EOS. As \autoref{cor:pepschemeexistence} asserts that no PEP scheme can exist on some domain if there exist pairs of points on an isobar with parallel but noncoincident tangent lines on that domain, we can check the existence visually by plotting a number of isobars in $(\rho,\varepsilon)$-coordinates. Nonexistence can then be verified by inspection if such a pair of points is present on an isobar e.g.\ before and after an inflection point of the isobar.

The results of this experiment are shown in \autoref{fig:isobarsn2} and \autoref{fig:isobarsco2}. Here, we visualize representative supercritical isobars in density-energy coordinates for nitrogen and carbon dioxide. For nitrogen, isobars are evaluated at \(P=5,10,15,25,35,\) and \(50\) MPa and plotted up to \(T=1000\) K; for carbon dioxide, we use \(P=8,10,15,25,35,\) and \(50\) MPa and \(T\leq1100\) K. Each isobar is sampled in temperature and mapped to density-energy coordinates. In addition to the supercritical isobars, the figures contain the critical point, the complete critical isobar \(p=p_{\mathrm{crit}}\), the complete critical isotherm \(T=T_{\mathrm{crit}}\), and the two-phase region obtained from the saturated-liquid and saturated-vapor lines. In the figures, the supercritical regime can be understood as the intersection of the regions above the critical isobar and critical isotherm. The separation between supercritical liquid-like and gas-like behavior is referred to as the Widom line \cite{guardonenonideal}. The Widom line is defined as the line of local maxima of the isobaric heat capacity \(c_p(\bm{\tau})\) connected to the critical point and is included to indicate the transition between liquid-like and gas-like supercritical states. For the carbon-dioxide reference EOS, the plotted EOS domain is additionally restricted to the published Span--Wagner admissible range \(T\leq1100\) K and \(p\leq800\) MPa, with states beyond the melting line excluded.

Visual inspection of \autoref{fig:isobarsn2} and \autoref{fig:isobarsco2} shows that supercritical real-gas isobars in density-energy coordinates are curved and have inflections and thus there are parallel noncoincident isobar tangent lines according to \autoref{thm:curvature}. Especially for supercritical isobars with pressures close to the critical pressure $p_{\mathrm{crit}}$ there seem to be gas-like states with low mass density $\rho$ and liquid-like states beyond the Widom line with high $\rho$ that have parallel tangent lines. For these pairs \autoref{cor:pepschemeexistence} asserts no PEP schemes exist and without the switch \eqref{eq:singularityswitch} the PEP schemes of \cite{channodal, coppolapressure, degrendeleconstruction} would be singular. Whereas for nitrogen in \autoref{fig:isobarsn2} the curvature and thus these singular pairs are more subtle, for carbon dioxide these pairs quite clearly exist.

\subsection{Numerical isobar search}
The previous experiment showed that for isobars with pressures close to the critical pressure $p_{\mathrm{crit}}$ the nonexistence of PEP schemes at gas-like/liquid-like interfaces is especially likely, hence we will inspect these isobars more closely. More specifically, we will perform a numerical search for pairs of points on these isobars with parallel and noncoincident tangent lines. This will give more conclusive numerical evidence for the nonexistence of PEP schemes for the EOS-molecule pairs that we consider. In particular, we are interested in transcritical pairs for which the two states lie on opposite sides of the Widom line, so that one state is liquid-like and the other gas-like because PEP schemes are specifically designed to treat material interfaces between such states.

\begin{figure}
    \centering
    \includegraphics[width=\linewidth]{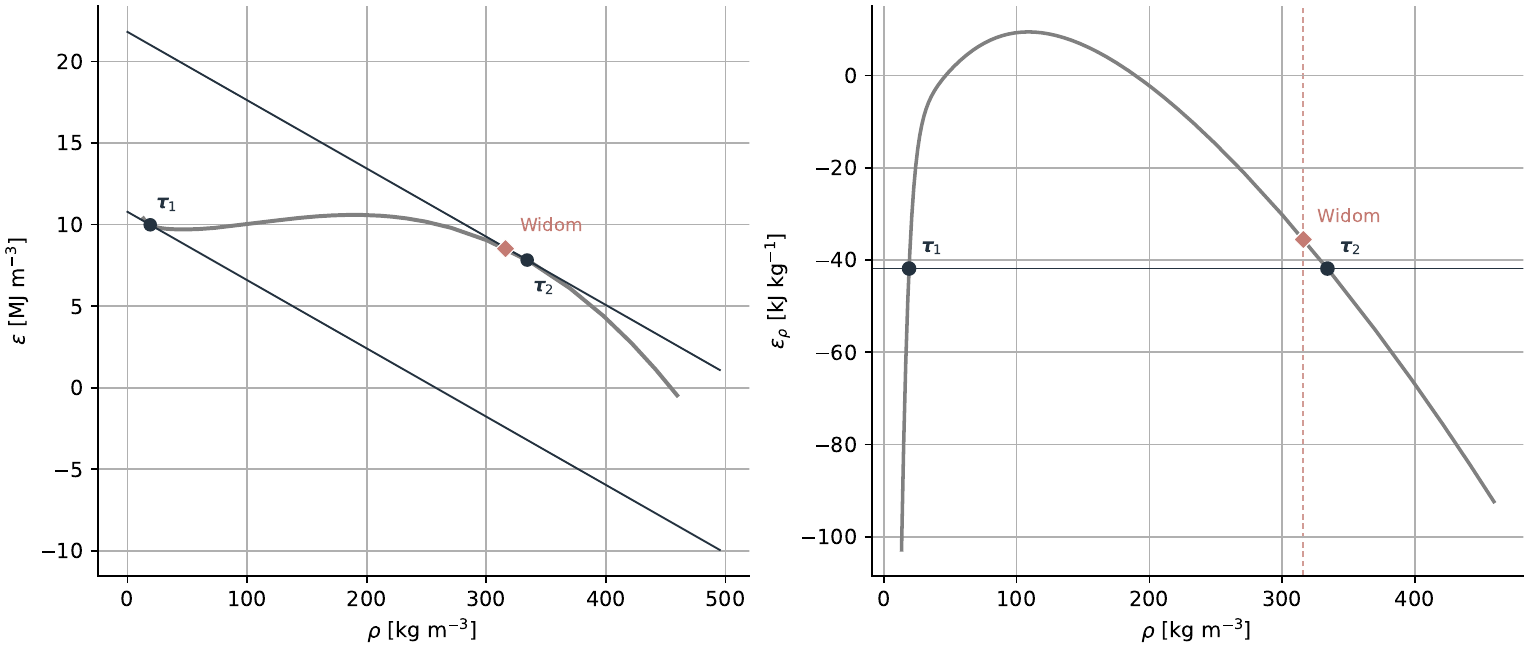}
    \caption{Parallel noncoincident tangent lines on the $P=4$ MPa nitrogen isobar computed using the Peng--Robinson EOS. Left: the isobar and tangent lines at the gas-like state $\bm{\tau}_1$ and liquid-like state $\bm{\tau}_2$. Right: the tangent slope $\varepsilon_\rho$ along the isobar, which is the same at $\bm{\tau}_1$ and $\bm{\tau}_2$. The Widom-line intersection is indicated in both panels.}
    \label{fig:parallelpairsn2}
\end{figure}

\begin{figure}
    \centering
    \includegraphics[width=\linewidth]{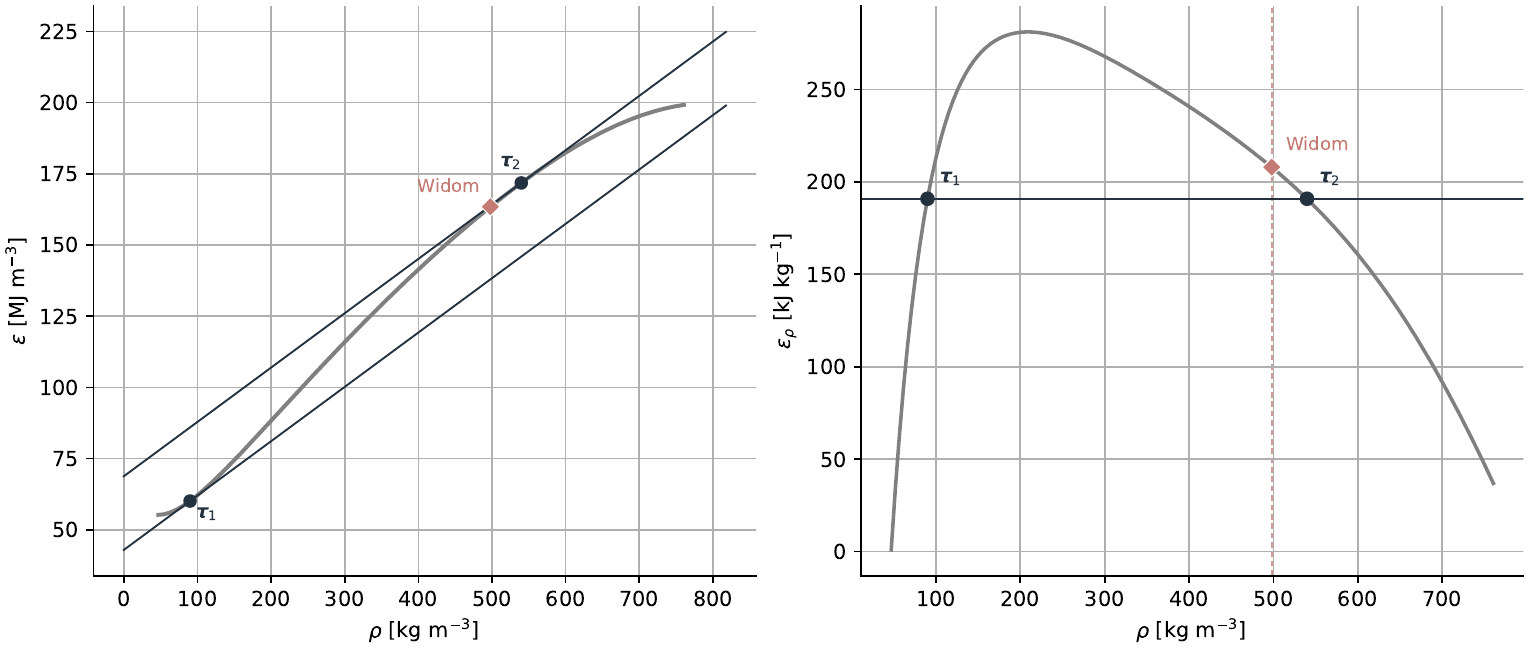}
    \caption{Parallel noncoincident tangent lines on the $P=10$ MPa carbon-dioxide isobar computed using the Span--Wagner EOS. Left: the isobar and tangent lines at the gas-like state $\bm{\tau}_1$ and liquid-like state $\bm{\tau}_2$. Right: the tangent slope $\varepsilon_\rho$ along the isobar, which is the same at $\bm{\tau}_1$ and $\bm{\tau}_2$. The Widom-line intersection is indicated in both panels.}
    \label{fig:parallelpairsco2}
\end{figure}

The results of this experiment are shown in \autoref{fig:parallelpairsn2} and \autoref{fig:parallelpairsco2}. The nitrogen experiment uses the Peng--Robinson EOS at \(P=4\) MPa, with the first state prescribed by \(\rho_1=19\ {\rm kg\,m^{-3}}\). Solving \(\rho(T)-\rho_1=0\) along the isobar gives \(T_1=699.7376836\) K. The carbon-dioxide experiment uses the HEOS Span--Wagner EOS at \(P=10\) MPa and prescribes \(\rho_1=90\ {\rm kg\,m^{-3}}\), giving \(T_1=599.6671554\) K. For each case, the isobar is first sampled at \(1400\) temperatures and the tangent slope \(\varepsilon_\rho=(\partial\varepsilon/\partial\rho)_p\) is evaluated from thermodynamic derivatives supplied by CoolProp. The sampled function \(\varepsilon_\rho(T)-\varepsilon_\rho(T_1)\) is then searched for sign changes away from the selected state, and each resulting bracket is refined using Brent's bracketed root-finding algorithm as implemented by scipy.optimize.brentq. This procedure gives an equal-slope companion point for nitrogen at \(\rho_2=334.1374674\ {\rm kg\,m^{-3}}\), \(T_2=129.3833791\) K, with \(\varepsilon_\rho=-41.8476990479\ {\rm kJ\,kg^{-1}}\) and a tangent-line-intercept difference of \(11.02730218\ {\rm MJ\,m^{-3}}\). For carbon dioxide the corresponding companion point is located at \(\rho_2=539.7113714\ {\rm kg\,m^{-3}}\), \(T_2=316.7150036\) K, with common tangent slope \(\varepsilon_\rho=190.857739525\ {\rm kJ\,kg^{-1}}\) and an intercept difference of \(25.85913376\ {\rm MJ\,m^{-3}}\). The resulting figures, \autoref{fig:parallelpairsn2} and \autoref{fig:parallelpairsco2}, show the two tangent lines in \((\rho,\varepsilon)\)-coordinates together with the tangent-slope function \(\varepsilon_\rho(\rho)\). The intersection of the selected isobar with the Widom line is included to indicate that the two selected states belong to opposite sides of the transcritical pseudo-phase transition.

This shows that \emph{no consistent, conservative and PEP schemes of the form \eqref{eq:discretization} exist} for nitrogen using the Peng--Robinson EOS and carbon dioxide using the Span--Wagner EOS if the respective supercritical isobars are included in the domain of interest of the scheme. 

\subsection{Supercritical material interfaces}
We have demonstrated that consistent, conservative PEP schemes do not always exist for real-gas applications. However, these schemes can still be `fixed' when used together with switches \eqref{eq:singularityswitch} to remove singularities. We refer to such a scheme as a conditional PEP scheme\footnote{Switches may cause loss of regularity of the spatial discretization \eqref{eq:discretization} and do not resolve regularity issues other than flux singularities, hence the semi-discretization \eqref{eq:discretization} may not always admit solutions. Thus we use the word `scheme' loosely throughout this section.}. In this case the resulting scheme can still suffer from stability issues. This happens even when the switches are not active and is caused by the tangent-line-intersection mechanism explained in \autoref{ssec:accuracyissues}. 

In this experiment we will analyze this mechanism numerically using two fully discrete cases based on supercritical material interfaces. The first case mimics an under-resolved material interface so that at the given grid resolution the physically smooth interface behaves essentially as a contact discontinuity. We will show that even if the PEP-mean construction using tangent-line intersections is quite far from singular, the construction can still lead to poor or even thermodynamically inadmissible means. Using these means can in turn cause strong spurious oscillations in mass density under practical $\mathrm{CFL}$ numbers and can even lead to negative mass densities in conditions far from a vacuum resulting in failure of the scheme. Importantly, this behavior is not confined to relatively small neighborhoods of an exact singularity, but can persist over extended portions of an isobar. The second case is a resolved supercritical material interface. This experiment will show that, remarkably, even when supercritical interfaces are resolved and smooth, conditional PEP schemes can encounter numerical problems due to the same tangent-line-intersection mechanism. We will compare with other nondissipative stable schemes for real gases. We will focus on a single fully discrete time step for both cases as the numerical instabilities occur immediately. Throughout, we will denote the updated state using a superscript $^+$ and the initial state without a superscript e.g.\ $\rho_j^+$ is the updated mass density at some grid cell $j$ while $\rho_j$ is the corresponding initial value.

In general, the experiments in this section show that it is not enough to merely switch to a different scheme when the PEP means become exactly singular. Near-singular configurations can already produce severely inaccurate or thermodynamically inadmissible PEP means, even when the material interface is smooth and well resolved.

\subsubsection{Under-resolved material interfaces}
This experiment tests the PEP-mean construction for under-resolved material interfaces under $(P,V)$-equilibrium in a fully discrete setting under a moderate $\mathrm{CFL}$ number. Some nonsingular under-resolved material interfaces will be constructed by perturbing an interface for which PEP means are singular. To deal with possible singularities we will use the switch proposed in \eqref{eq:singularityswitch}. The performance of the conditional PEP scheme is compared against a reference scheme for the nondissipative simulation of supercritical fluids.

For this experiment we consider carbon dioxide described by the Span--Wagner EOS \cite{SpanWagner1996} through the CoolProp HEOS implementation \cite{Bell2014CoolProp} on an isobar with \(P=10\) MPa. The results of the experiment are shown in \autoref{fig:underresolvedinterfaceco2}, \autoref{fig:tangentintersectionco2} and \autoref{tab:pverrorsco2}. This experiment is performed on a small grid of eight equally spaced nodes at \(x_j/h=-3.5 + j \) with $j = 0,1,\dots,7$. The initial condition of this experiment is a mass-density discontinuity at $x = 0$ representing an under-resolved supercritical material interface:
\begin{equation*}
    \rho_j = \begin{cases}
        \rho_L & x_j < 0 \\
        \rho_R & x_j \geq 0 
    \end{cases}, \qquad p_j = P, \qquad v_j = V,
\end{equation*}
where \(P=10\) MPa and \(V=100\ {\rm m\,s^{-1}}\). The left state is fixed at the gas-like state from the singular pair identified previously, \(\rho_L=90\ {\rm kg\,m^{-3}}\), while the dense liquid-like reference state is \(\rho_0=539.7113714\ {\rm kg\,m^{-3}}\). Eight perturbed right states are generated according to:
\begin{equation*}
    \rho_R=\rho_0+\delta\rho_0,\qquad \delta\rho_0=0.03 \mu\,\rho_0,\qquad \mu\in\{-5,-3,-2,-1,1,2,3,5\},
\end{equation*}
which gives approximately \(\delta\rho_0\approx\{-81,-49,-32, -16,16,32,49,81\}\ {\rm kg\,m^{-3}}\). Define the relative slope difference as:
\begin{equation*}
    \eta(\rho_1,\rho_2) := \frac{|\varepsilon_{\rho}(\rho_2) - \varepsilon_{\rho}(\rho_1)|}{\max(|\varepsilon_{\rho}(\rho_2)|,|\varepsilon_{\rho}(\rho_1)|)}.
\end{equation*}
Across the material interface, $\eta(\rho_L,\rho_R)$ ranges from $0.0351103$ to $0.221497$, so the switch \eqref{eq:singularityswitch} is not activated and the PEP flux is used. For each prescribed density, the temperature $T$ and internal-energy density $\varepsilon$ are obtained on the \(10\) MPa isobar using the same density-to-temperature inversion as in the preceding isobar experiment. The conservative variables are initialized as \(\bm{U}=(\rho,\rho V,\varepsilon+\tfrac12\rho V^2)\). The PEP flux uses the unique tangent-line intersection \((\tilde\rho,\tilde\varepsilon)\) of the two interface states, so that on the $(P,V)$-equilibrium set the numerical flux is \(f_\rho=\tilde\rho\, V\), \(f_m=Vf_\rho+P\), and \(f_E=V(\tilde\varepsilon+\tfrac12Vf_\rho+P)\). As a nondissipative reference scheme for general EOS we use the entropy-conserving and kinetic-energy-consistent KEEP-DG flux with the symmetrized Itoh--Abe discrete gradient as proposed in \cite{kleingeneralized}. Both schemes are advanced by one forward-Euler stage with a single common time step for all eight cases, prescribed by:
\begin{equation*}
\mathrm{CFL}_{\max} =\frac{\Delta t}{h}\max_{\delta \rho_0, j}(|V|+c_{j}^{(\delta \rho_0)})=0.5,
\end{equation*}
where $c_{j}^{(\delta \rho_0)}$ is the speed of sound at $x_j$ for the $\rho_R$ computed with perturbation $\delta \rho_0$. This gives \(\Delta t/h\approx 1.1\times10^{-3}\ {\rm s\,m^{-1}}\) and \(V\Delta t/h\approx0.11\). The reported pressure- and velocity-equilibrium errors are the relative quantities:
\begin{equation}
    \mathcal{E}_{p,\mathrm{rel}} :=\frac{\max_j|p_j^{+}-P|}{P}, \qquad \mathcal{E}_{v,\mathrm{rel}} :=\frac{\max_j|v_j^{+}-V|}{|V|}, 
    \label{eq:equilibriumerrors}
\end{equation}
with pressure $p_j^{+}$ reconstructed from the updated conservative state $\bm{U}_j^{+}$ using the CoolProp \((\rho,e)\) flash and velocity $v_j^{+}$ evaluated directly as \(m_j^{+}/\rho_j^{+}\). The relative errors are tabulated in \autoref{tab:pverrorsco2}. If pressure cannot be evaluated at some node due to inadmissible $(\rho,e)$-values, that configuration is reported as thermodynamically inadmissible in \autoref{tab:pverrorsco2}. The initial conditions and forward-Euler updates of the mass-density profile on the grid are plotted in \autoref{fig:underresolvedinterfaceco2}. The tangent-line construction of the mass-density and internal-energy-density means at the discontinuity is shown in \autoref{fig:tangentintersectionco2}. For visualization purposes the left panel in \autoref{fig:tangentintersectionco2} is detrended by the secant line $\ell_{\mathrm{sec}}$ through $\bm{\tau}_{\mathrm{fix}}$ and $\bm{\tau}_{0}$ and the right panel is detrended by the tangent line $\ell_{\mathrm{fix}}$ at $\bm{\tau}_{\mathrm{fix}}$.

The intersections in \autoref{fig:tangentintersectionco2} yield negative density means for $\delta\rho_0<0$ and means exceeding both input densities for $\delta\rho_0>0$. Their effect is apparent from the updates of the two cells immediately adjacent to the discontinuity. The mass flux across the material interface is $\tilde{\rho}_{R,L}\,V$, where $\tilde{\rho}_{R,L}$ denotes the tangent-line-intersection mean across the interface, while consistency gives $\rho_L\,V$ and $\rho_R\,V$ at their other interfaces. Writing $\rho_{\mathrm{left}}^{+}$ and $\rho_{\mathrm{right}}^{+}$ for the updated densities in the cells immediately to the left and right of the interface, respectively, we obtain:
\begin{equation*}
\rho_{\mathrm{left}}^{+}
=\rho_L-\lambda(\tilde{\rho}_{R,L}-\rho_L),
\qquad
\rho_{\mathrm{right}}^{+}
=\rho_R+\lambda(\tilde{\rho}_{R,L}-\rho_R),
\qquad
\lambda:=\frac{V\Delta t}{h}.
\end{equation*}
For $\delta\rho_0<0$, the negative mean gives a negative interface mass flux despite the positive physical velocity $V > 0$. Mass is consequently transferred towards the gas-like region on the left and its density increases, while the density in the neighboring liquid-like cell decreases. This explains the distorted profiles in \autoref{fig:underresolvedinterfaceco2}.

\begin{figure}
    \centering
    \includegraphics[width=\linewidth]{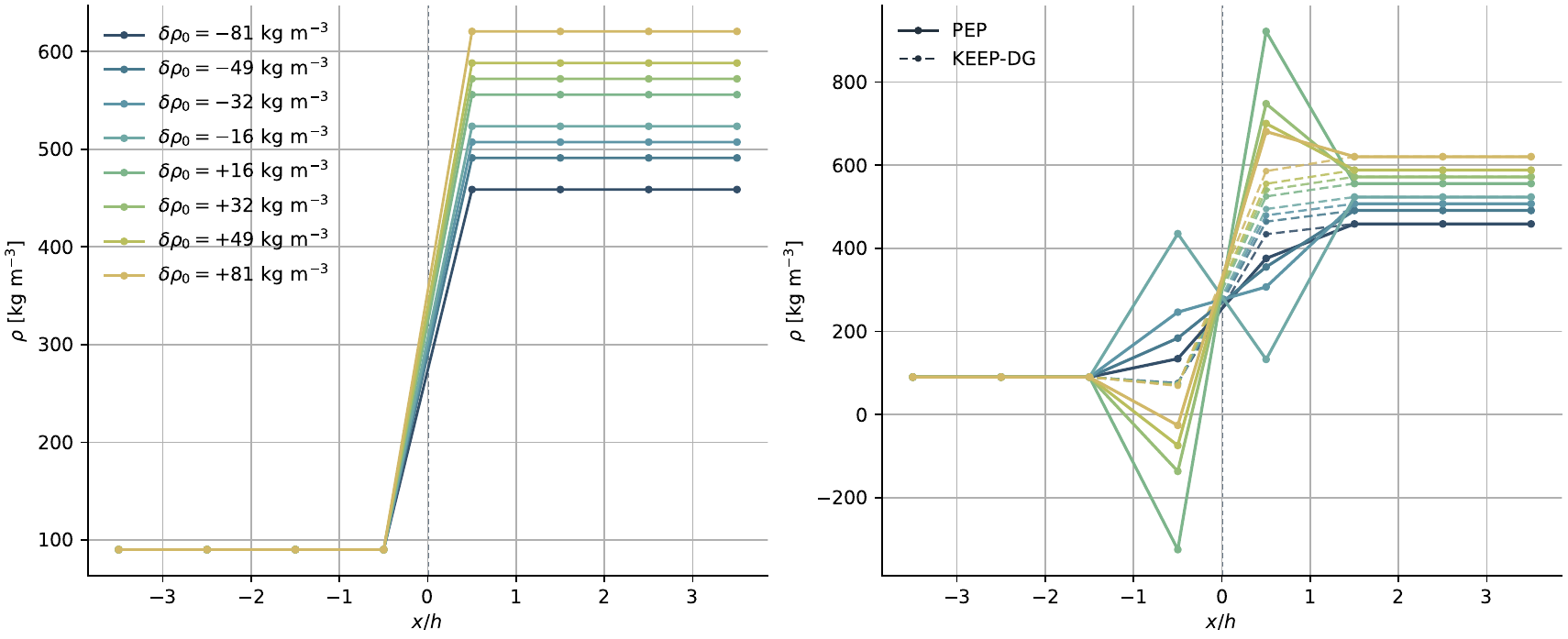}
    \caption{Left: the initial condition of the under-resolved supercritical carbon-dioxide material interfaces at \(P=10\) MPa. Right: one-step evolution using forward Euler and the conditional PEP and reference KEEP-DG schemes.}
    \label{fig:underresolvedinterfaceco2}
\end{figure}

\begin{figure}
    \centering
    \includegraphics[width=\linewidth]{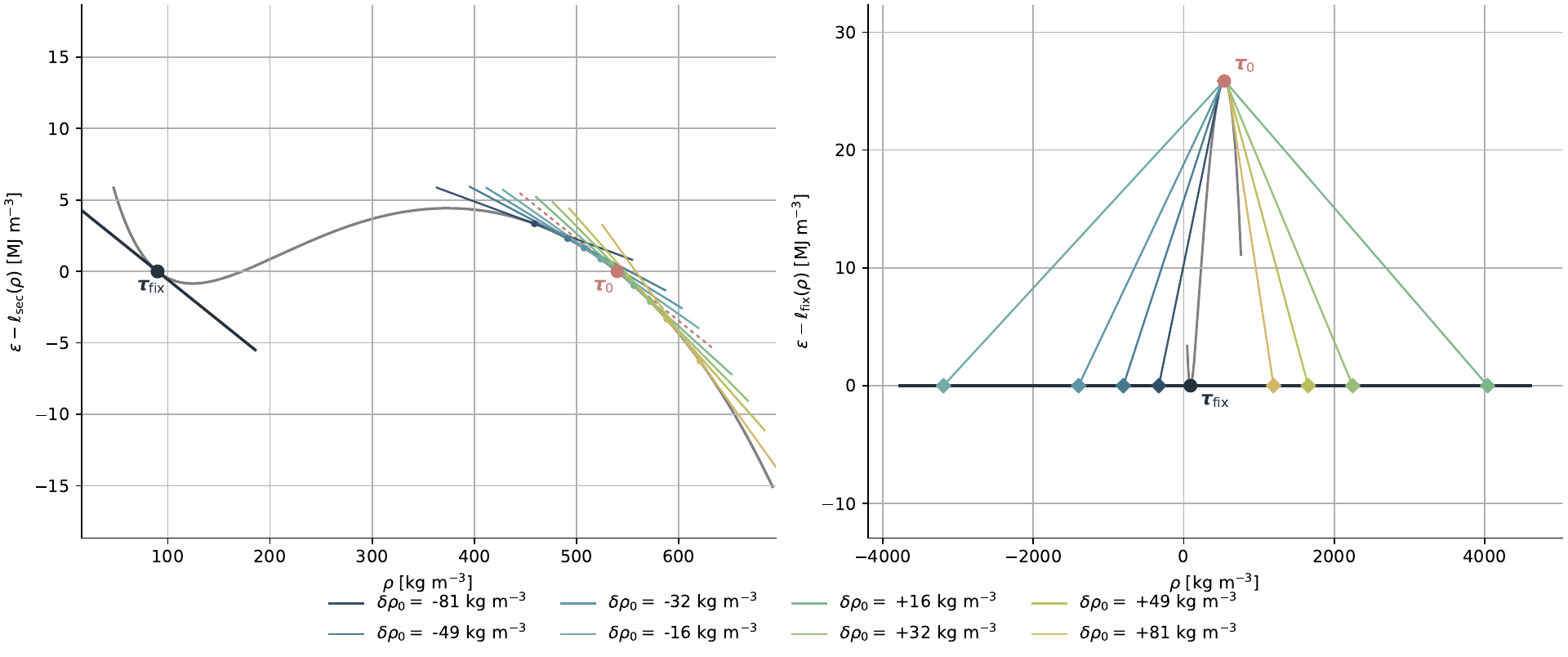}
    \caption{Left: isobar tangent lines of the perturbed states in the under-resolved supercritical carbon-dioxide material interfaces. Right: the intersection points of the tangent lines in the left panel.}
    \label{fig:tangentintersectionco2}
\end{figure}

For $\delta\rho_0>0$, the large positive mean instead produces excessive mass transfer to the right. The left cell loses mass because its outgoing flux $\tilde{\rho}_{R,L}\, V$ exceeds its incoming flux $\rho_L\,V$, while the right cell gains mass. In all $\delta \rho_0 > 0$ cases, the net loss over the chosen time step exceeds the mass initially present in the left cell, giving $\rho_{\mathrm{left}}^{+}<0$, as seen in \autoref{fig:underresolvedinterfaceco2}. This occurs even though the physical advective CFL number satisfies $\lambda\approx0.11<1$. In comparison, the KEEP-DG scheme \cite{kleingeneralized} is less oscillatory and produces no negative densities in these tests.

Note that the conditional PEP scheme is only PEP semi-discretely. Comparing the fully discrete relative $(P,V)$-equilibrium-preservation errors in \autoref{tab:pverrorsco2} shows that KEEP-DG tends to give better fully discrete pressure-equilibrium-preservation behavior for this case than the conditional PEP scheme. Due to the negative densities $\rho_j^{+}$ for $\delta \rho_0 > 0$ the pressure $p_j^{+}$ cannot even be computed for these cases. Nonetheless, both schemes preserve discrete velocity equilibrium.

\begin{table}
\centering
\begin{tabular}{ccccc}
\hline
$\delta\rho_0$ [kg m$^{-3}$] & $\mathcal{E}_{p,\mathrm{rel}}^{\mathrm{PEP}}$ & $\mathcal{E}_{p,\mathrm{rel}}^{\mathrm{KEEP}}$ & $\mathcal{E}_{v,\mathrm{rel}}^{\mathrm{PEP}}$ & $\mathcal{E}_{v,\mathrm{rel}}^{\mathrm{KEEP}}$ \\
\hline
-80.96 & $5.014\times 10^{-2}$ & $8.230\times 10^{-2}$ & $1.421\times 10^{-16}$ & $2.842\times 10^{-16}$ \\
-48.57 & $1.852\times 10^{-1}$ & $9.285\times 10^{-2}$ & $1.421\times 10^{-16}$ & $2.842\times 10^{-16}$ \\
-32.38 & $3.582\times 10^{-1}$ & $9.831\times 10^{-2}$ & $1.421\times 10^{-16}$ & $2.842\times 10^{-16}$ \\
-16.19 & $5.655\times 10^{-1}$ & $1.039\times 10^{-1}$ & $0$ & $2.842\times 10^{-16}$ \\
16.19 & $\mathrm{inadmissible}$ & $1.155\times 10^{-1}$ & $1.421\times 10^{-16}$ & $1.421\times 10^{-16}$ \\
32.38 & $\mathrm{inadmissible}$ & $1.215\times 10^{-1}$ & $2.842\times 10^{-16}$ & $1.421\times 10^{-16}$ \\
48.57 & $\mathrm{inadmissible}$ & $1.277\times 10^{-1}$ & $2.842\times 10^{-16}$ & $0$ \\
80.96 & $\mathrm{inadmissible}$ & $1.404\times 10^{-1}$ & $7.105\times 10^{-16}$ & $1.421\times 10^{-16}$ \\
\hline
\end{tabular}
\caption{Relative pressure-equilibrium and velocity-equilibrium errors of the conditional PEP and reference KEEP-DG schemes for the different perturbed states. The phrase `inadmissible' indicates that the $(\rho,e)$ flash problem to compute the pressure in a grid cell did not have a solution.}
\label{tab:pverrorsco2}
\end{table}

\begin{figure}
    \centering
    \includegraphics[width=\linewidth]{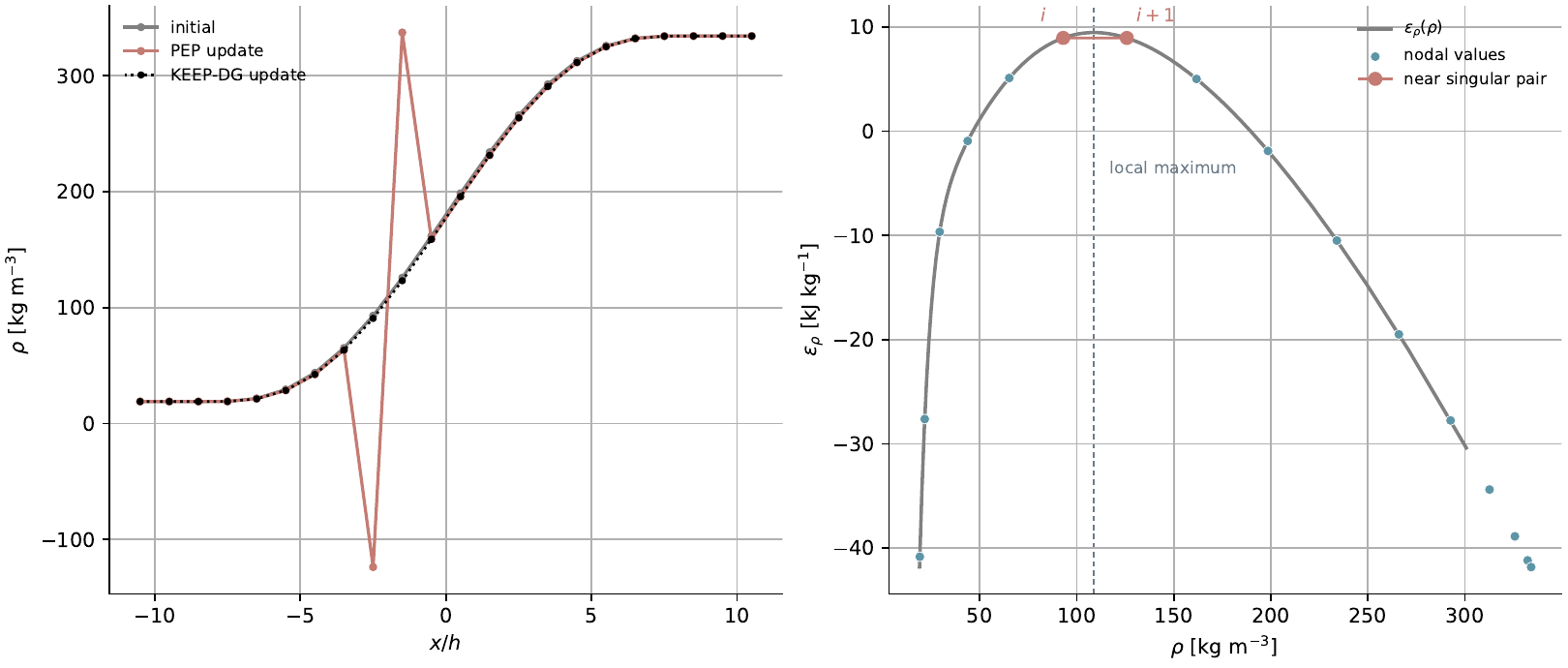}
    \caption{Left: one-step evolution of a resolved supercritical nitrogen material interface at \(P=4\) MPa using forward Euler and the conditional PEP and reference KEEP-DG schemes. Right: the nodal values of the slope $\varepsilon_{\rho}$ around the resolved material interface.}
    \label{fig:resolvedinterfacen2}
\end{figure}

\begin{figure}
    \centering
    \includegraphics[width=\linewidth]{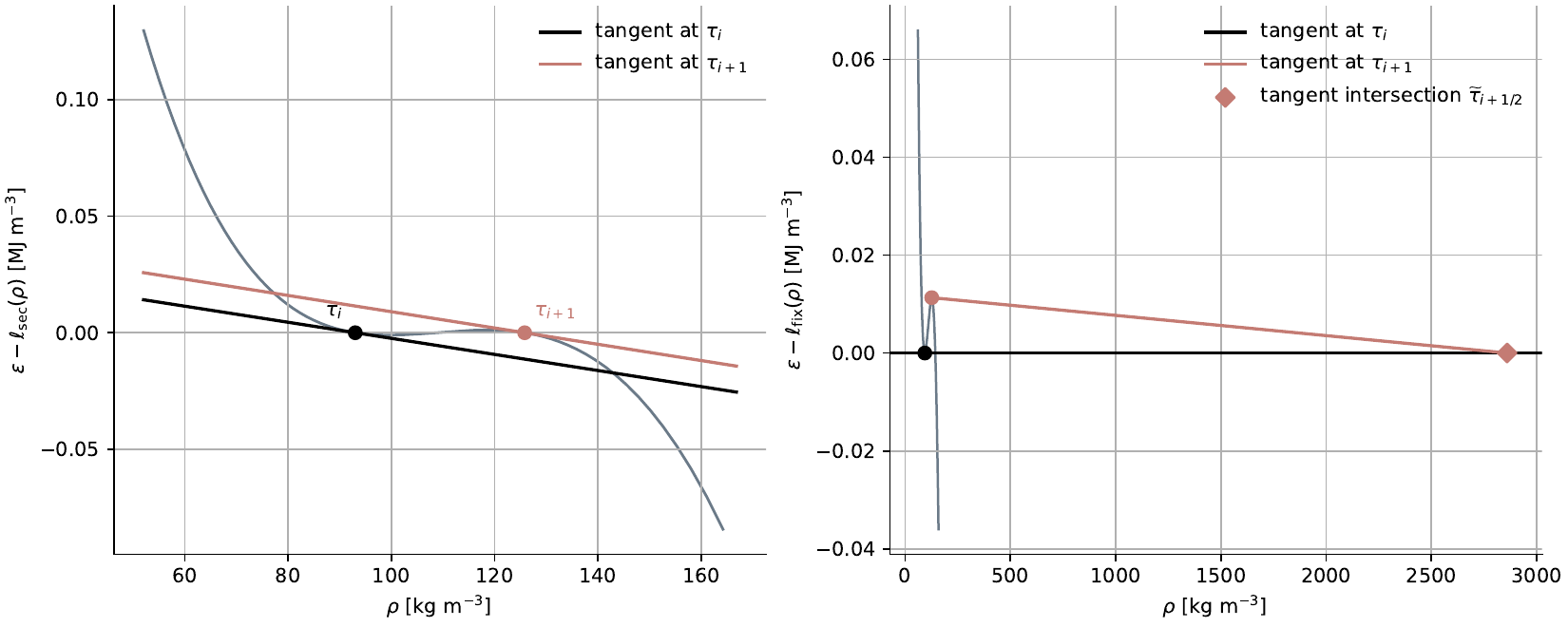}
    \caption{Left: isobar tangent lines of the near-singular pair of neighboring grid nodes in the resolved supercritical nitrogen material interface. Right: the intersection point of the tangent lines in the left panel.}
    \label{fig:tangentintersectionn2}
\end{figure}

\begin{table}
\centering
\begin{tabular}{lcc}
\hline
scheme & $\mathcal{E}_{p,\mathrm{rel}}$ & $\mathcal{E}_{v,\mathrm{rel}}$ \\
\hline
PEP & $\mathrm{inadmissible}$ & $1.421\times 10^{-16}$ \\
KEEP-DG & $4.997\times 10^{-4}$ & $2.842\times 10^{-16}$ \\
\hline
\end{tabular}
\caption{Relative pressure-equilibrium and velocity-equilibrium errors of the conditional PEP and reference KEEP-DG schemes. The phrase `inadmissible' indicates that the $(\rho,e)$ flash problem to compute the pressure in a grid cell did not have a solution.}
\label{tab:pverrorsn2}
\end{table}

\subsubsection{Resolved material interfaces}
Even when the material interface is well resolved, numerical problems can still occur for conditional PEP schemes for real gases. In particular, using a smooth supercritical material interface that is resolved by the grid, this experiment will show that the same failure mechanism can occur for states around a local maximum in the tangent-line slope $\varepsilon_{\rho}(\rho)$ i.e.\ the isobar inflections discussed in \autoref{ssec:curvaturethm}. 

For the resolved experiment we consider nitrogen described by the Peng--Robinson EOS \cite{PengRobinson1976, BellJager2016} through CoolProp \cite{Bell2014CoolProp} on the \(P=4\) MPa isobar. The results of this experiment are shown in \autoref{fig:resolvedinterfacen2}, \autoref{fig:tangentintersectionn2} and \autoref{tab:pverrorsn2}. The initial condition is given by a smooth transition in the mass density in a $(P,V)$-equilibrium set from $\rho_L$ to $\rho_R$. We use the singular nitrogen pair from the preceding experiment as the two asymptotic states, \(\rho_L=19\ {\rm kg\,m^{-3}}\) and \(\rho_R=334.1374674\ {\rm kg\,m^{-3}}\). The grid consists of \(N=256\) equally spaced nodes \(x_j/h=j-(N-1)/2\), \(j=0,\ldots,255\). Specifically, the initial density is prescribed as the compact \(C^2\) transition function and the pressure and velocity are constant:
\begin{gather*}
    \rho_j = \rho_L+(\rho_R-\rho_L) S\!\left(\frac{x_j/h-\phi+8}{16}\right), \qquad S(q)= \begin{cases} 0,&q\le0,\\ 6q^5-15q^4+10q^3,&0<q<1,\\ 1,&q\ge1, \end{cases} \\
    p_j = P, \qquad v_j = V,
\end{gather*}
with $P = 4$ MPa and \(V=100\ {\rm m\,s^{-1}}\). The transition has width \(16h\) and shift parameter \(\phi=-0.09703312053\). The transition is shifted, not its width. The shift is designed to induce an instability on the grid. Namely, with this shift, the local maximum of \(\varepsilon_\rho(\rho)\) lies between the neighboring nodes with indices $i$ and $i+1$ at \(x_i/h=-2.5\) and \(x_{i+1}/h= -1.5\), whose densities are \(\rho_i=93.0199707\) and \(\rho_{i+1}=125.8094368\ {\rm kg\,m^{-3}}\), with tangent slopes \(\varepsilon_{\rho,i}=8.950035373\) and \(\varepsilon_{\rho,i+1}=8.945897923\ {\rm kJ\,kg^{-1}}\), respectively, so that $\eta(\rho_i,\rho_{i+1}) \approx 4.62283\times 10^{-4}$ and is well above the \(100\epsilon_{\rm mach} \approx 2.22045 \times 10^{-14}\) singularity-switch threshold \eqref{eq:singularityswitch} used in the calculation. The PEP flux is evaluated in the same way as in the under-resolved experiment using the tangent-line-intersection mean, with ghost states equal to the corresponding endpoints of the transition. The same initial condition is also advanced with the nondissipative KEEP-DG flux described above \cite{kleingeneralized}. Both calculations use one forward-Euler stage and the same time step:
\begin{equation*}
    \mathrm{CFL}_{\max} =\frac{\Delta t}{h}\max_j(|V|+c_j)=0.5.
\end{equation*}
This gives \(\Delta t/h \approx 7.8\times10^{-4}\ {\rm s\,m^{-1}}\) and \(V\Delta t/h\approx 0.078\). The relative pressure and velocity errors are computed with the same definitions \eqref{eq:equilibriumerrors} as in the under-resolved experiment and are tabulated in \autoref{tab:pverrorsn2}. In addition, the spatial profile of the mass-density update zoomed in around the material interface and the geometry of the near-singular neighboring pair are visualized in \autoref{fig:resolvedinterfacen2} and \autoref{fig:tangentintersectionn2}, respectively. Similarly to the previous experiment, the first panel of \autoref{fig:tangentintersectionn2} is detrended by the secant $\ell_{\mathrm{sec}}$ through the two nodal states \(\bm{\tau}_i, \bm{\tau}_{i+1}\), while the second is detrended by the black tangent $\ell_{\mathrm{fix}}$ so that the far-away intersection $(\tilde{\rho},\tilde{\varepsilon})$ defining the PEP mean is displayed explicitly.

As shown in \autoref{fig:tangentintersectionn2}, the density mean is again overpredicted $\tilde{\rho}_{i+1,i} \gg \rho_i , \rho_{i+1}$ by the conditional PEP scheme. This overprediction is the result of the fact that the isobar tangent-line slope $\varepsilon_{\rho}$ has a local maximum between the states in cells $i$ and $i+1$, and the shift $\phi$ is deliberately chosen so that $C \epsilon_{\mathrm{mach}} \ll \eta(\rho_{i},\rho_{i+1}) \ll 1$ as can be seen in \autoref{fig:resolvedinterfacen2}. This results in nearly parallel tangent lines at the states in cells $i$ and $i+1$ for which the PEP mean is not rejected by the singularity switch \eqref{eq:singularityswitch}. The same issue as in the positive-density-perturbation cases of the under-resolved experiment now occurs between the nodal values in grid cells $i$ and $i+1$. The net mass loss during the time step exceeds the mass initially present in the cell, even though the advective $\mathrm{CFL}$ number is less than one $V\Delta t / h \approx 0.078 < 1$. The KEEP-DG scheme does not suffer from these issues. Moreover, like in the previous experiment the fully discrete relative PEP error as tabulated in \autoref{tab:pverrorsn2} shows that KEEP-DG outperforms the exact conditional PEP scheme due to the negative mass densities $\rho_j^{+} < 0$. 

In \autoref{sec:largetolerance} we briefly investigate the use of larger relative tolerances $C\epsilon_{\text{mach}}$. For these tolerances it is still possible to find shifts $\phi$ for which errors produced by the PEP scheme are substantially larger than those produced by KEEP-DG.

\section{Conclusion}\label{sec:conclusion}
In this work, we have provided a complete existence characterization for consistent and algebraically pressure-equilibrium-preserving (PEP) numerical-flux functions for general equations of state. The main contribution is a characterization of algebraic pressure-equilibrium preservation in density-energy coordinates, based on geometric properties of the equation of state. Our geometric condition avoids the additional thermodynamic-derivative assumptions required by previous algebraic compatibility conditions \cite{channodal, terashimaapproximately}. Specifically, we showed that such a consistent and algebraically PEP numerical-flux function exists on an admissible supercritical regime if and only if the tangent lines at every pair of points on each isobar in that regime have a nonempty intersection. This condition is both necessary and sufficient: necessity follows directly from pressure and velocity equilibrium, while sufficiency is established constructively through a numerical flux whose mass-density and internal-energy-density means are chosen from the corresponding tangent-line intersections. Whenever the isobar can be parameterized by mass density, this geometric condition recovers the existing algebraic PEP compatibility condition of \cite{channodal} and gives it a direct geometric interpretation. The tangent-line intersection condition can be checked visually when plotting isobars, but involves checking pairs of points. Under some additional conditions it is possible to only check the isobar for local geometric conditions, like the presence of inflections. This was discussed in more detail in \autoref{ssec:curvaturethm}. 

The numerical experiments demonstrate that this existence condition is restrictive for practically relevant real-gas equations of state \cite{Bell2014CoolProp}. For nitrogen described by the Peng--Robinson EOS \cite{PengRobinson1976, BellJager2016} and carbon dioxide described by the Span--Wagner EOS \cite{SpanWagner1996}, we identified transcritical pairs of states with parallel and noncoincident isobar tangents, showing numerically that no consistent, conservative and exactly PEP scheme of the considered form can exist on regimes containing these states. More importantly, the practical difficulties are not confined to the exactly singular configurations. Near-parallel tangent lines may cause the unique PEP mean to lie far from the two interface states and can therefore produce severely inaccurate or thermodynamically inadmissible updated states. Such near-singular behavior need not be confined to small neighborhoods of isolated singular pairs, but may affect substantial portions of an isobar. A switch-based remedy that is made sufficiently aggressive to exclude these problematic configurations may therefore deactivate the PEP flux over large parts of the relevant state space, effectively leaving the scheme pressure-equilibrium-preserving only on a restricted subset of configurations. In both under-resolved and smooth resolved material-interface experiments, this mechanism produced strong density oscillations and, in some cases, negative mass densities despite moderate CFL numbers. The comparison with the nondissipative KEEP-DG scheme \cite{kleingeneralized} further illustrates that exact semi-discrete pressure-equilibrium preservation does not necessarily translate into a more accurate or robust fully discrete approximation for the test cases and type of discretization \eqref{eq:discretization} we considered. These results heavily suggest that if, in addition to consistency and numerical conservation, the pressure-equilibrium-preservation property is desired for numerical schemes for supercritical fluids, novel approaches are required. 

Our analysis has focused on simple conservative discretizations using two-point numerical-flux functions in one spatial dimension. A natural extension to high-order accuracy would be via split-form discretization \cite{ranochapreventing}. The assumptions on the EOS are quite minimal and mainly represent fundamental thermodynamic principles \cite{menikoffriemann} which many viable EOS for compressible fluid dynamics satisfy. Finally, we have analyzed exact semi-discrete PEP, our analysis has not covered fully-discrete or approximately PEP formulations \cite{terashimaapproximately}.

\section*{CRediT authorship contribution statement}
\textbf{R.B. Klein}: Conceptualization, Methodology, Software, Formal analysis, Investigation, Validation, Visualization, Writing – original draft, Writing – review \& editing. 

\section*{Software and reproducibility statement}
The code used to generate the figures and reproduce all numerical values in the main body of the text is made available at \cite{reprorepo}.

\section*{Declaration of Generative AI and AI-assisted technologies in the writing process}
ChatGPT 5.6 Sol and ChatGPT 6 Astra were used to assist in the writing of this manuscript and in the formalization of some proofs of intermediate results. ChatGPT 5.6 Sol and ChatGPT 6 Astra wrote the code for this manuscript. An independent ChatGPT 6 Astra instance was used to review the manuscript and code. After using these tools/services, the author reviewed and edited the content as needed and takes full responsibility for the content of the publication.

\section*{Declaration of competing interests}
The author declares that there are no known competing financial interests or personal relationships that could have appeared to influence the work reported in this paper.

\section*{Data availability}
No data to declare. 

\section*{Acknowledgements}
The author gratefully acknowledges the funding for this project obtained from Delft University of Technology. The author also thanks Benjamin Sanderse for proofreading the manuscript which lead to a more concise presentation and the additional results on isobar inflections and Ruud Henkes, Rene Pecnik and Pedro Costa for the fruitful discussions.

\appendix
\section{Isobar inflection points}
Isobar inflections are points where PEP schemes can fail. We discuss their physical meaning. We also prove a result related to them that makes demonstrating nonexistence of PEP schemes more convenient under some reasonable additional conditions on the isobar.

\subsection{Physical discussion of real-gas isobar inflections}\label{ssec:physicalinflection}
Using mass density $\rho$ and temperature $T$ as natural variables, a Van-der-Waals (VdW) fluid with a constant isochoric heat capacity $c_v \in \mathbb{R}_+$ satisfies:
\begin{equation*}
    p(\rho,T) = \frac{\rho R T}{1 - b\rho} - a \rho^2, \qquad \varepsilon(\rho,T) = \rho c_vT - a\rho^2,
\end{equation*}
where $R \in \mathbb{R}_+$ is the specific gas constant of the fluid, the coefficient $b \in \mathbb{R}_+$ is used to model the strength of repulsive forces and referred to as the excluded-volume coefficient and the coefficient $a \in \mathbb{R}_+$ models the strength of attraction forces between molecules and is called the attractive coefficient. On an isobar $p(\rho,T) = P$ holds for some $P \in \mathbb{R}_+$ and we can solve for the temperature along the isobar:
\begin{equation*}
    T(\rho) = \frac{(P + a \rho^2)(1-b\rho)}{\rho R}.
\end{equation*}
By eliminating temperature, the internal-energy density on an isobar can then be expressed in terms of the mass density alone:
\begin{align*}
    \varepsilon(\rho) := \varepsilon(\rho, T(\rho)) &= \frac{c_v}{R}(P + a \rho^2)(1-b\rho) - a \rho^2 \\
    &= \left(-ab \frac{c_v}{R} \right)\rho^3  + \left(a \frac{c_v}{R}-a\right)\rho^2  + \left(-bP \frac{c_v}{R}\right) \rho + \left(P\frac{c_v}{R}\right).
\end{align*}
For increasing $\rho$, this function demonstrates how the internal-energy density $\varepsilon$ changes during an isobaric compression. The rate at which the energy changes during the compression is given by the derivative $\varepsilon_\rho$. Inflection points are extrema of $\varepsilon_\rho(\rho)$ where $\varepsilon_{\rho\rho}(\rho) = 0$. These derivatives can be computed as follows:
\begin{align*}
    \varepsilon_{\rho}(\rho) &= \left(-3ab \frac{c_v}{R} \right)\rho^2  + \left(2a \frac{c_v}{R}-2a\right)\rho  + \left(-bP \frac{c_v}{R}\right), \\
    \varepsilon_{\rho\rho}(\rho) &= \left(-6ab \frac{c_v}{R} \right)\rho + \left(2a \frac{c_v}{R}-2a\right).
\end{align*}
For $c_v/R > 1$ it holds that $\left(-3abc_v/R \right) < 0$ and $\left(2ac_v/R-2a\right) > 0$. Thus for sufficiently small $\rho > 0$ the linear term in $\varepsilon_\rho$ dominates the quadratic term and the rate of change of internal-energy $\varepsilon_{\rho}$ during the compression is increasing. However, for larger $\rho$ the negative quadratic term overtakes the linear term and $\varepsilon_\rho$ subsequently starts decreasing. This results in a maximum of $\varepsilon_\rho$ where $\varepsilon_{\rho\rho} = 0$ and in an inflection of $\varepsilon$. 

While the quadratic term $\left(-3ab c_v/R \right)\rho^2$ is a coupling between both attractive and repulsive contributions, the linear term $\left(2a c_v/R-2a\right)\rho$ consists of purely attracting contributions. Note that if $b = 0$ there would be no coupling contribution and the rate $\varepsilon_\rho$ would keep increasing resulting in no inflection, while if $a = 0$ the rate would be constant at $\left(-bPc_v/R\right)$. This shows that indeed both effects are needed to induce an isobar inflection in VdW fluids and illustrates that in general for real-gas EOS inflections can occur when multiple physical mechanisms are modeled.

\subsection{An inflection theorem}\label{ssec:curvaturethm}
Whenever it is possible to write an isobar as a $C^2$ graph $\varepsilon(\rho)$ over mass density, the geometric condition in \autoref{thm:pepexistence} can be simplified to some extent. Namely, in this case the lack of tangent-line intersections is equivalent to $\varepsilon_{\rho\rho}(\rho)$ crossing zero. This is stated in the following theorem with a proof in \autoref{ssec:curvaturecriterion}.
\begin{thm}[Isobaric second-derivative criterion]\label{thm:curvature}
    Let $\Gamma_P$ be a nonempty, connected $C^2$ isobar in an admissible supercritical regime $\mathcal{S}_{\bm{\tau}}$, with $p_\varepsilon\neq0$ on $\Gamma_P$. Then there exist an open interval $I\subset\mathbb{R}_+$ and $\varepsilon\in C^2(I)$ such that:
    \begin{equation*}
        \Gamma_P=\{(\rho,\varepsilon(\rho)):\rho\in I\}.
    \end{equation*}
    Moreover, $\Gamma_P$ has two disjoint tangent lines if and only if $\varepsilon_{\rho\rho}$ takes both positive and negative values.
\end{thm}
Such a zero crossing can result in inflections in the isobar as depicted in \autoref{fig:paraiso}. However, isolated inflections are not the only case that can occur. If the zero crossing is stretched over an interval then the isobar instead has a locally linear stretched inflection as depicted in \autoref{fig:isobarcases}. Both cases will give rise to infinite singular pairs of disjoint tangent lines in a neighborhood around the, possibly stretched, inflection. The existence of such singular pairs can thus easily be verified by plotting the second derivative $\varepsilon_{\rho\rho}(\rho)$ along an isobar and inspecting for zero crossings. Isobar inflections tend to lie in the gas-like regime away from transcritial pseudo-phase transitions. For the purpose of simulating transcritical flows it might, however, be more interesting to check whether more distant transcritical singular pairs exist with states in gas-like and liquid-like regimes of the isobar. Since such pairs tend to show up around under-resolved transcritical interfaces in simulations. In this case a direct tangent-line-based search might be more suitable. We investigate both types of singular pairs in our numerical experiments.

The $p_{\varepsilon} \neq 0$ condition in \autoref{thm:curvature} was necessary to prevent the second mechanism for constructing disjoint tangent lines from taking place. This second mechanism is the isobar turning back on itself as depicted in \autoref{fig:isobarcases}. As the curvature of such an isobar curve does not change sign the suggested condition in \autoref{thm:curvature} does not hold for this case, while it clearly has disjoint tangent lines. Such a case could be detected by computing the curvature integral over curve sections and identifying whether it exceeds $\pi$ over any section. Since this is a highly nonlocal operation it is not particularly convenient and a direct tangent-line-based search might again be more suitable. Finally, a singular inflection might occur where at the inflection $p_{\varepsilon} = 0$ holds. This case is also not accounted for by \autoref{thm:curvature}.

\begin{figure}
    \centering
    \includegraphics[width=\linewidth]{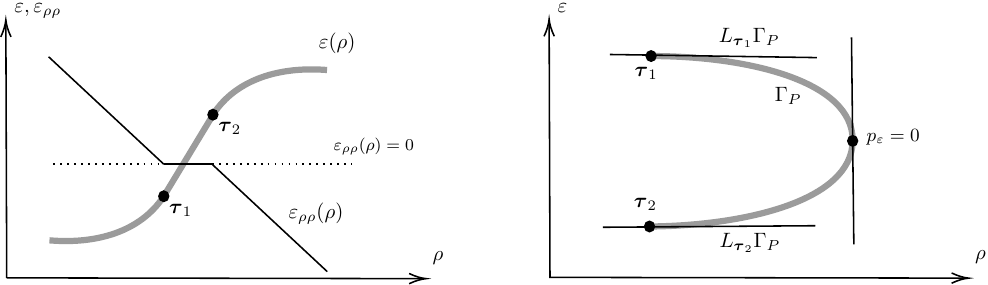}
    \caption{Impressions of different isobar geometry types. Left: an isobar with a locally linear stretched inflection and a graph structure on $\rho$. An overlay of the second derivative $\varepsilon_{\rho\rho}$ is also sketched. Right: an isobar that turns back onto itself resulting in at least one point where $p_\varepsilon = 0$. Thus the isobar cannot be defined as graph on $\rho$. }
    \label{fig:isobarcases}
\end{figure}

\section{Large-tolerance behavior}\label{sec:largetolerance}
To investigate larger switching tolerances, we repeat the resolved
material-interface experiment with shifts $\phi=-0.09624928$,
$\phi=-0.098199541$, $\phi=-0.118$ and $\phi=-0.322$, using
$C\epsilon_{\mathrm{mach}}=10^{-4}$, $10^{-3}$, $0.01$ and $0.10$,
respectively. All other parameters remain unchanged, and arithmetic
means of density and internal-energy density are used wherever the
switch activates. We measure density accuracy by:
\begin{equation*}
    \mathcal{E}_{\rho,\mathrm{rel}}
    :=\max_j
    \frac{|\rho_j^{+}-\rho_{\mathrm{ex},j}^{+}|}
         {\rho_{\mathrm{ex},j}^{+}},
    \qquad
    \rho_{\mathrm{ex},j}^{+}
    :=\rho_{\mathrm{init}}(x_j-V\Delta t),
\end{equation*}
where $\rho_{\mathrm{init}}$ is the continuous initial density profile
and $\rho_{\mathrm{ex},j}^{+}$ is its exact advected value at node $j$
after one time step. For the neighboring nodes $i,i+1$ straddling the
maximum of $\varepsilon_\rho$, the relative slope differences
$\eta(\rho_i,\rho_{i+1})$ slightly exceed the respective switching
tolerances, so the PEP flux remains active at these interfaces.
As shown in \autoref{tab:resolvedlargeswitch}, density errors of
approximately $1092\%$, $109\%$, $10.9\%$ and $1.22\%$ persist
despite the enlarged tolerances. KEEP-DG produces substantially
smaller density errors in all four cases.

\begin{table}[h!]
\centering
\begin{tabular}{ccccc}
\hline
$C\epsilon_{\mathrm{mach}}$
& $\phi$
& $\eta(\rho_i,\rho_{i+1})$
& $\mathcal{E}_{\rho,\mathrm{rel}}^{\mathrm{PEP}}$
& $\mathcal{E}_{\rho,\mathrm{rel}}^{\mathrm{KEEP}}$ \\
\hline
$0.01\%$ & $-0.09624928$ & $0.010010\%$ & $1091.9\%$ & $0.26335\%$ \\
$0.1\%$ & $-0.098199541$ & $0.100100\%$ & $109.26\%$ & $0.26333\%$ \\
$1\%$  & $-0.118$ & $1.0103\%$  & $10.894\%$ & $0.26304\%$ \\
$10\%$ & $-0.322$ & $10.0065\%$ & $1.2243\%$ & $0.25032\%$ \\
\hline
\end{tabular}
\caption{Relative density errors for the resolved nitrogen material
interface with switching tolerances of $0.01\%$, $0.1\%$, $1\%$
and $10\%$. All errors are measured after one forward-Euler update
at $\mathrm{CFL}_{\max}=0.5$. Both methods retain positive densities
in the $1\%$ and $10\%$ cases. For $0.01\%$ and $0.1\%$, the
conditional PEP updates have minimum densities of approximately
$-899$ and $-8.39$ ${\rm kg\,m^{-3}}$, respectively, while KEEP-DG
retains positive densities.}
\label{tab:resolvedlargeswitch}
\end{table}

\section{Proofs}
\subsection{Density-energy diffeomorphism}\label{ssec:densityenergydiffeomorphism}
The proof of \autoref{prop:densityenergydiffeomorphism} is given as:
\begin{proof}
    Because $\nu > 0$ on $\mathcal{T}$ and $e \in C^2(\mathcal{S}_{\bm{\eta}})$ by the regularity assumption in \autoref{dfn:admissiblesupercriticalregime} it holds that $\bm{\Phi} \in C^2(\mathcal{S}_{\bm{\eta}})$. By \eqref{eq:thermodynamicstatespace}, \eqref{eq:pressuretemperature} and the monotonicity property in \autoref{dfn:admissiblesupercriticalregime}, the Jacobian $D\bm{\Phi}$ of $\bm{\Phi}$ has a determinant:
    \begin{equation*}
        \det D\bm{\Phi}(\bm{\eta}) = - \frac{T(\bm{\eta})}{\nu^3} < 0,  \qquad \forall \bm{\eta} \in \mathcal{S}_{\bm{\eta}}.
    \end{equation*}
    As we have shown $\bm{\Phi} \in C^2(\mathcal{S}_{\bm{\eta}})$, by the inverse-function theorem $\bm{\Phi}$ is therefore a local $C^2$ diffeomorphism. If it is also injective, $\bm{\Phi}$ is then a $C^2$ diffeomorphism onto its image $\mathcal{S}_{\bm{\tau}}$ \cite{guillemindifferential}. We will show $\bm{\Phi}$ is injective by contradiction. Assume $\bm{\Phi}$ is not injective. Then there exist $(\nu_1,\sigma_1), (\nu_2,\sigma_2) \in \mathcal{S}_{\bm{\eta}}$ for which $(\nu_1,\sigma_1) \neq (\nu_2,\sigma_2)$ as points in $\mathbb{R}^2$, so that $\bm{\Phi}(\nu_1,\sigma_1) = \bm{\Phi}(\nu_2,\sigma_2)$. Then $1/\nu_1 = 1/\nu_2$, so $\nu_1 = \nu_2 = \nu$. Also $e(\nu,\sigma_1) / \nu = e(\nu,\sigma_2) /\nu$, so $e(\nu,\sigma_1) = e(\nu,\sigma_2)$. But by the monotonicity condition $e$ is injective at fixed $\nu$, thus $\sigma_1 = \sigma_2$; a contradiction. Thus $\bm{\Phi}$ is an injective local $C^2$ diffeomorphism and therefore a $C^2$ diffeomorphism onto its image.
\end{proof}

\subsection{Pressure regularity}\label{ssec:pressureregularity}
For clarity we add a subscript with the variables to the gradient operator throughout the proof of \autoref{prop:pressureregularity}. The proof of \autoref{prop:pressureregularity} is given as:
\begin{proof}
    Denote pressure in the original $(\nu,\sigma)$-coordinates by $\tilde{p}$. In density-energy coordinates the pressure is given by $p(\bm{\tau}):=\tilde{p}\bigl(\bm{\Phi}^{-1}(\bm{\tau})\bigr)$. Since $\tilde{p}=-e_\nu \in C^1(\mathcal{S}_{\bm{\eta}})$ and $\bm{\Phi}^{-1}\in C^2(\mathcal{S}_{\bm{\tau}})$ by \autoref{prop:densityenergydiffeomorphism}, it follows that
    $p\in C^1(\mathcal{S}_{\bm{\tau}})$. Moreover, strict stability gives:
    \begin{equation*}
        \tilde{p}_\nu(\bm{\eta})=-e_{\nu\nu}(\bm{\eta}) <0, \qquad \forall \bm{\eta}\in\mathcal{S}_{\bm{\eta}},
    \end{equation*}
    and hence $\nabla_{\bm{\eta}}\, \tilde{p}(\bm{\eta})\neq\bm{0}$. Since $\tilde{p}(\bm{\eta})=p\bigl(\bm{\Phi}(\bm{\eta})\bigr)$ the chain rule yields:
    \begin{equation*}
        \nabla_{\bm{\eta}}\, \tilde{p}(\bm{\eta})
        =
        D\bm{\Phi}(\bm{\eta})^{T}
        \nabla_{\bm{\tau}}\, p\bigl(\bm{\Phi}(\bm{\eta})\bigr).
    \end{equation*}
    As $D\bm{\Phi}(\bm{\eta})$ is nonsingular and $\nabla_{\bm{\eta}}\, \tilde{p}(\bm{\eta})\neq\bm{0}$, the gradient $\nabla_{\bm{\tau}}\, p(\bm{\Phi}(\bm{\eta}))$ cannot vanish. Since $\bm{\Phi}$ maps $\mathcal{S}_{\bm{\eta}}$ onto $\mathcal{S}_{\bm{\tau}}$, the result follows.
\end{proof}

\subsection{Isobar curves}\label{ssec:isobarcurves}
The proof of \autoref{prop:isobarcurves} is given as:
\begin{proof}
    Fix some $P \in \mathbb{R}_+$ so that $\Gamma_P$ is nonempty. Let $\bm{\tau}_0 \in \Gamma_P$ be arbitrary. According to \autoref{prop:pressureregularity} at least one component of the gradient is nonzero, thus either $p_{\rho}(\bm{\tau}_0) \neq 0$ or $p_{\varepsilon}(\bm{\tau}_0) \neq 0$ holds. Assume the case $p_{\varepsilon}(\bm{\tau}_0) \neq 0$ holds. Define the function $F : \mathcal{S}_{\bm{\tau}} \rightarrow \mathbb{R}$ as:
    \begin{equation*}
        F(\bm{\tau}) := p(\bm{\tau}) - P.
    \end{equation*}
    By \autoref{prop:pressureregularity} it holds that $F \in C^1(\mathcal{S}_{\bm{\tau}})$ and:
    \begin{equation*}
        F_{\varepsilon}(\bm{\tau}_0) = p_\varepsilon(\bm{\tau}_0) \neq 0, \qquad F(\bm{\tau}_0) = 0.
    \end{equation*}
    Denote $\bm{\tau}_0 =: (\rho_0,\varepsilon_0)$. The implicit-function theorem \cite{marsdenvector} can be applied and asserts the existence of open intervals $I, K \subset \mathbb{R}$ such that $\bm{\tau}_0 \in I \times K \subset \mathcal{S}_{\bm{\tau}}$ and a function $\psi : I \rightarrow K$ such that $\psi \in C^1(I)$ and so that for all $\rho \in I$ there exists a unique $\psi(\rho) \in K$ so that:
    \begin{equation*}
        F(\rho,\psi(\rho)) = 0.
    \end{equation*}
    By uniqueness and the facts that $\bm{\tau}_0 \in I \times K$ and $F(\bm{\tau}_0) = 0$, it must hold that $\psi(\rho_0) = \varepsilon_0$. Define $J := I$ and the map:
    \begin{equation*}
        \bm{\varphi} : I \rightarrow \mathbb{R}^2, \qquad \bm{\varphi}(\rho) := \begin{bmatrix}
            \rho \\ \psi(\rho)
        \end{bmatrix}.
    \end{equation*}
    Since $\psi(\rho_0) = \varepsilon_0$, choosing $t_0 = \rho_0$ gives $\bm{\varphi}(\rho_0) = \bm{\tau}_0$. It also holds that $\bm{\varphi} \in C^1(I)$ and since $\tfrac{d\rho}{d\rho} = 1$, we have $\bm{\varphi}'(\rho) \neq 0$ for all $\rho \in I$, thus the map is regular. It remains to show that $\bm{\varphi}(I) = \Gamma_P \cap (J \times K)$ and that it has a $C^1$ inverse on its image. By the implicit-function theorem, if $\rho \in I = J$, then $\psi(\rho) \in K$ thus $\bm{\varphi}(\rho) \in J\times K \subset \mathcal{S}_{\bm{\tau}}$. Furthermore, $F(\bm{\varphi}(\rho)) = 0$ so $\bm{\varphi}(\rho) \in \Gamma_P$ by definition. Thus $\bm{\varphi}(I) \subseteq \Gamma_P \cap (J\times K)$. Now assume $(\rho,\varepsilon) \in \Gamma_P \cap (J \times K)$. Then $F(\rho,\varepsilon) = 0$ and $\rho \in I = J$, while $\varepsilon \in K$. But $\psi(\rho) \in K$ is the unique value so that $F(\bm{\varphi}(\rho)) = F(\rho,\psi(\rho))=0$. Thus $\psi(\rho) = \varepsilon$ must hold. Then $(\rho,\varepsilon) \in \bm{\varphi}(I)$, so $\bm{\varphi}(I) = \Gamma_P \cap (J\times K)$. Finally, the inverse of $\bm{\varphi}$ for $\bm{\tau} \in \bm{\varphi}(I)$ is simply the $C^1$ projection $\bm{\varphi}^{-1}(\bm{\tau}) = \rho$.

    The proof for the case $p_{\rho}(\bm{\tau}_0) \neq 0$ is completely analogous.
\end{proof}

\subsection{Tangent-space orthogonality}\label{ssec:tangentspaceorthogonality}
The proof of \autoref{lem:tangentspaceorthogonality} is given as:
\begin{proof}
    Let $P \in \mathbb{R}_+$ so that $\Gamma_P$ is nonempty. Let $\bm{\tau}_0 := (\rho_0,\varepsilon_0) \in \Gamma_P$ be arbitrary. Due to \autoref{prop:pressureregularity}, either $p_{\rho}(\bm{\tau}_0) \neq 0$ or $p_{\varepsilon}(\bm{\tau}_0) \neq 0$ holds necessarily. Assume the latter. Repeating the proof of \autoref{prop:isobarcurves}, there exist open intervals $I,K \subset \mathbb{R}$ so that $\bm{\tau}_0 \in I \times K \subset \mathcal{S}_{\bm{\tau}}$ and a function $\psi : I \rightarrow K$ for which $\psi \in C^1(I)$ and the following holds for any $\rho \in I$:
    \begin{equation*}
        p(\rho,\psi(\rho)) - P = 0.
    \end{equation*}
    Moreover, for all $\rho \in I$ it holds that $p_{\varepsilon}(\rho,\psi(\rho)) \neq 0$. Using implicit differentiation we have:
    \begin{equation*}
        \psi'(\rho) = -\frac{p_{\rho}(\rho,\psi(\rho))}{p_{\varepsilon}(\rho,\psi(\rho))}.
    \end{equation*}
    Since, in the proof of \autoref{prop:isobarcurves} the parameterization $\bm{\varphi}$ was defined as $\bm{\varphi}(\rho) = (\rho, \psi(\rho))$ and $T_{\bm{\tau}_0}\Gamma_P := \operatorname{span}(\bm{\varphi}'(\rho_0))$, we have by definition:
    \begin{equation*}
        T_{\bm{\tau}_0}\Gamma_P := \operatorname{span}\left(\begin{bmatrix}
            1 \\ \displaystyle-\frac{p_{\rho}(\bm{\tau}_0)}{p_{\varepsilon}(\bm{\tau}_0)}
        \end{bmatrix}\right).
    \end{equation*}
    Since $p_{\varepsilon}(\bm{\tau}_0) \neq 0$, the first equality in \eqref{eq:tangentspaceidentity} follows after multiplying by $-p_{\varepsilon}(\bm{\tau}_0)$. The same reasoning can be repeated for the $p_{\rho}(\bm{\tau}_0)\neq 0$ case.
    
    Assume $\bm{q} \in \operatorname{span}((-p_{\varepsilon}(\bm{\tau}_0),p_{\rho}(\bm{\tau}_0)))$, so $\bm{q} = c \cdot (-p_{\varepsilon}(\bm{\tau}_0),p_{\rho}(\bm{\tau}_0))$ for some $c \in \mathbb{R}$ as we have just demonstrated. Then:
    \begin{equation*}
        \left<\nabla p(\bm{\tau}_0), \bm{q} \right> = c \cdot (- p_{\rho}(\bm{\tau}_0) p_{\varepsilon}(\bm{\tau}_0) + p_{\varepsilon}(\bm{\tau}_0) p_{\rho}(\bm{\tau}_0)) = 0.
    \end{equation*}
    Thus $T_{\bm{\tau}_0}\Gamma_P \subseteq \ker\left(\left<\nabla p(\bm{\tau}_0),\cdot \right>\right)$. Now let $\bm{q} := (q_1,q_2) \in \ker\left(\left<\nabla p(\bm{\tau}_0),\cdot \right>\right)$. By \autoref{prop:pressureregularity} either $p_{\rho}(\bm{\tau}_0) \neq 0$ or $p_{\varepsilon}(\bm{\tau}_0) \neq 0$ holds necessarily. Again, assume the latter. Then:
    \begin{equation*}
        0 = \left<\nabla p(\bm{\tau}_0), \bm{q} \right> = p_{\rho}(\bm{\tau}_0) q_1 + p_{\varepsilon}(\bm{\tau}_0)q_2.
    \end{equation*}
    Since $p_{\varepsilon}(\bm{\tau}_0) \neq 0$, we have:
    \begin{equation*}
        q_2 = -\left(\frac{p_{\rho}(\bm{\tau}_0)}{p_{\varepsilon}(\bm{\tau}_0)}\right) q_1.
    \end{equation*} 
    Thus $\bm{q} \in \operatorname{span}\left(\left(1, -\left(\frac{p_{\rho}(\bm{\tau}_0)}{p_{\varepsilon}(\bm{\tau}_0)}\right)\right)\right) = T_{\bm{\tau}_0}\Gamma_P$ by our previous reasoning. Thus $T_{\bm{\tau}_0}\Gamma_P = \ker\left(\left<\nabla p(\bm{\tau}_0),\cdot \right>\right)$. Similar arguments can be made for the $p_{\rho}(\bm{\tau}_0) \neq 0$ case.
\end{proof}

\subsection{Supercritical conservative variables}\label{ssec:supercriticalconservativevariables}
The proof of \autoref{prop:supercriticalconservativevariables} is given as:
\begin{proof}
    It holds that $\bm{\Theta} \in C^\infty(\mathcal{S}_{\bm{\tau}}\times \mathbb{R})$. Since $\rho>0$ on $\mathcal{S}_{\bm{\tau}}$, it also holds that $\rho > 0$ on $\mathcal{S}_{\bm{U}} = \bm{\Theta}(\mathcal{S}_{\bm{\tau}} \times \mathbb{R})$, since $\bm{\Theta}$ is the identity in its first component. Then one readily verifies that:
    \begin{equation}
        \bm{\Theta}^{-1}(\bm{U})
        =
        \bigl(\bm{\tau}(\bm{U}),v(\bm{U})\bigr),
        \label{eq:thetainverse}
    \end{equation}
    defines a smooth inverse on $\mathcal{S}_{\bm{U}}$. Hence $\bm{\Theta}$ is a smooth
    diffeomorphism. For brevity, denote:
    \begin{equation*}
        \tilde{\mathcal{S}}_U := \{(\rho,m,E)\in \mathbb{R}^3\,:\,\rho>0,\,\bm{\tau}(\rho,m,E)\in \mathcal{S}_{\bm{\tau}} \}.
    \end{equation*}
    Assume $\bm{U} \in \mathcal{S}_{\bm{U}}$. Then there exists some $\bm{V} := (\bm{\tau},v) \in \mathcal{S}_{\bm{\tau}} \times \mathbb{R}$ such that $\bm{\Theta}(\bm{V}) = \bm{U}$. Therefore $(\bm{\tau},v) = \bm{V} = \bm{\Theta}^{-1}(\bm{\Theta}(\bm{V})) = \bm{\Theta}^{-1}(\bm{U}) = (\bm{\tau}(\bm{U}),v(\bm{U}))$. Since $\bm{\tau} \in \mathcal{S}_{\bm{\tau}}$, we have $\bm{\tau}(\bm{U})\in \mathcal{S}_{\bm{\tau}}$. Thus $\mathcal{S}_{\bm{U}} \subseteq \tilde{\mathcal{S}}_U$. Now let $\bm{U} := (\rho,m,E)\in \tilde{\mathcal{S}}_U$. Then $\rho > 0$, so $v(\bm{U}) = m/\rho \in \mathbb{R}$. Also, $\bm{\tau}(\bm{U}) \in \mathcal{S}_{\bm{\tau}}$. Then $(\bm{\tau}(\bm{U}),v(\bm{U}))\in \mathcal{S}_{\bm{\tau}}\times \mathbb{R}$. But direct substitution in $\bm{\Theta}$ gives $\bm{U} =\bm{\Theta}(\bm{\tau}(\bm{U}),v(\bm{U})) \in \bm{\Theta}(\mathcal{S}_{\bm{\tau}}\times \mathbb{R})$. Thus $\mathcal{S}_{\bm{U}} = \tilde{\mathcal{S}}_U$.
\end{proof}

\subsection{Equilibrium set}\label{ssec:equilibriumset}
The proof of \autoref{prop:equilibriumset} is given as:
\begin{proof}
    Both sets are empty if the isobar is empty, thus we can assume the isobar is nonempty. Let $\bm{U} \in \mathcal{M}_{P,V}$. Then $\bm{\Theta}^{-1}(\bm{U}) = (\bm{\tau}(\bm{U}),v(\bm{U})) \in \mathcal{S}_{\bm{\tau}} \times \mathbb{R}$ by \autoref{prop:supercriticalconservativevariables}. Moreover, $p(\bm{\tau}(\bm{U})) = P$, thus $\bm{\tau}(\bm{U}) \in \Gamma_P$. Also, $v(\bm{U}) = V$. Then, since $\bm{\Theta}(\bm{\tau}(\bm{U}), V) = \bm{\Theta}(\bm{\tau}(\bm{U}), v(\bm{U})) = \bm{\Theta}(\bm{\Theta}^{-1}(\bm{U})) = \bm{U}$ and $\bm{\tau}(\bm{U}) \in \Gamma_P$, $\bm{U} \in \{ \bm{\Theta}(\bm{\tau},V)\, :\, \bm{\tau}\in \Gamma_P\}$. So $\mathcal{M}_{P,V} \subseteq \{ \bm{\Theta}(\bm{\tau},V)\, :\, \bm{\tau}\in \Gamma_P\}$. Now let $\bm{U} \in \{ \bm{\Theta}(\bm{\tau},V)\, :\, \bm{\tau}\in \Gamma_P\}$. Then, $\bm{U} = \bm{\Theta}(\bm{\tau},V)$ for some $\bm{\tau} \in \mathcal{S}_{\bm{\tau}}$, thus $\bm{U} \in \mathcal{S}_{\bm{U}}$ by \autoref{prop:supercriticalconservativevariables}. Moreover, $(\bm{\tau}(\bm{U}),v(\bm{U})) = \bm{\Theta}^{-1}(\bm{U}) = \bm{\Theta}^{-1}(\bm{\Theta}(\bm{\tau},V)) = (\bm{\tau},V)$. Since $\bm{\tau} \in \Gamma_P$, we have $p(\bm{\tau}(\bm{U})) =p(\bm{\tau}) = P$. Therefore $\bm{U} \in \mathcal{M}_{P,V}$ and $\mathcal{M}_{P,V} = \{ \bm{\Theta}(\bm{\tau},V)\, :\, \bm{\tau}\in \Gamma_P\}$.
\end{proof}

\subsection{Flux consistency}\label{ssec:fluxconsistency}
The proof of \autoref{prop:fluxconsistency} is given as:
\begin{proof}
    We will first prove $\bm{f}_h$ is well-defined. Let $\bm{U}_R,\bm{U}_L \in \mathcal{S}_{\bm{U}}$ be arbitrary so that $P = p(\bm{\tau}_R) = p(\bm{\tau}_L)$ with $\bm{\tau}_R := \bm{\tau}(\bm{U}_R)$ and $\bm{\tau}_L := \bm{\tau}(\bm{U}_L)$. Thus $\bm{\tau}_R, \bm{\tau}_L \in \Gamma_P$ and so by assumption $L_{\bm{\tau}_L}\Gamma_P \cap L_{\bm{\tau}_R}\Gamma_P \neq \varnothing$. Since $L_{\bm{\tau}_L}\Gamma_P$ and $L_{\bm{\tau}_R}\Gamma_P$ are affine subspaces in $\mathbb{R}^2$ with nonempty intersection, their intersection is either a unique point or the entire affine subspace formed by either tangent line. Either way, the minimization in \eqref{eq:massenergymean} has a unique solution. When the pressures differ, \eqref{eq:massenergymean} selects the baseline mean $\bm{\tau}^*$, which is well-defined by assumption.
\\
    We now show consistency. Let $\bm{U}_R=\bm{U}_L =\bm{U} \in \mathcal{S}_{\bm{U}}$. Denote $\bm{\tau} := \bm{\tau}(\bm{U})$. It holds that $\bm{\tau}_R = \bm{\tau}_L = \bm{\tau}$ and $p(\bm{\tau}_R) = p(\bm{\tau}_L) = p(\bm{\tau}) = P$. Since both points are the same, $L_{\bm{\tau}_R}\Gamma_P \cap L_{\bm{\tau}_L}\Gamma_P = L_{\bm{\tau}}\Gamma_P$. By definition $\bm{\tau} \in L_{\bm{\tau}}\Gamma_P$. Moreover, by consistency $\bm{\tau}^*(\bm{U},\bm{U}) = \bm{\tau}(\bm{U}) = \bm{\tau}$. The unique minimizer in \eqref{eq:massenergymean} is therefore $\bm{\tau} = (\rho,\varepsilon)$. By consistency of the other means in \eqref{eq:flux} the entire flux $\bm{f}_h$ is consistent.
\end{proof}

\subsection{Isobaric second-derivative criterion}\label{ssec:curvaturecriterion}
The proof of \autoref{thm:curvature} is given as:
\begin{proof}
    We will first show that $\Gamma_P$ must globally be a graph $\varepsilon(\rho)$ on $\rho$. Assume there exist two points $\bm{\tau}_1 = (\rho_1,\varepsilon_1),\bm{\tau}_2 = (\rho_2, \varepsilon_2) \in \Gamma_P$ with the same $\rho_1=\rho_2=\hat{\rho}$ coordinate, but distinct $\varepsilon_1 \neq \varepsilon_2$. Since $\Gamma_P$ is connected and $C^2$, by \cite[Corollary 3.4.6]{bergerdifferential}, there exists a regular $C^1$ curve $\bm{\gamma} = (\gamma^\rho, \gamma^\varepsilon) : [0,1] \rightarrow \Gamma_P$ so that $\bm{\gamma}_1(0) = \bm{\tau}_1$ and $\bm{\gamma}(1) = \bm{\tau}_2$. By Rolle's theorem there exists some $t^* \in (0,1)$ so that:
    \begin{equation*}
        \gamma_t^\rho(t^*) = 0.
    \end{equation*}
    Since this is the $\rho$ component of a tangent vector and the curve parameterization is regular, this implies $p_\varepsilon(\bm{\gamma}(t^*)) = 0$ by \autoref{lem:tangentspaceorthogonality}, a contradiction. Thus there exist no two distinct points on a connected $\Gamma_P$ with the same $\rho$ value. The map $\Gamma_P \ni (\rho, \varepsilon)  \mapsto \rho$ is therefore injective and by connectedness its image is an open interval $I$. The local parameterizations $\psi(\rho)$ in the proof of \autoref{prop:isobarcurves} define a global $C^2$ graph $\varepsilon(\rho)$ on $I$ so that $\Gamma_P = \{(\rho,\varepsilon(\rho)):\rho\in I\}$.

    We now prove the disjoint tangent-line characterization. Necessity is shown be contraposition. If $\varepsilon_{\rho\rho}$ does not take both signs, $\varepsilon_{\rho}$ is monotone. Tangent line slopes either are equal or unequal. Monotonocity in the equality case, $\varepsilon_{\rho}(\rho_1)=\varepsilon_{\rho}(\rho_2)$ for $\rho_1<\rho_2$, forces $\varepsilon_{\rho}$ to be constant on $[\rho_1,\rho_2]$, so all parallel tangents coincide. Tangent lines with different slopes, $\varepsilon_{\rho}(\rho_1)\neq\varepsilon_{\rho}(\rho_2)$, intersect at a unique point, so in either case the intersection is nonempty. We will only sketch the sufficiency argument. If $\varepsilon_{\rho\rho}$ takes both signs, $\varepsilon_{\rho}$ is not monotone, and there exist $\rho_1<\hat{\rho}<\rho_2$ with $\varepsilon_{\rho}(\hat{\rho})>\max\{\varepsilon_{\rho}(\rho_1),\varepsilon_{\rho}(\rho_2)\}$ or the reversed inequalities. In the former case, choose $\max\{\varepsilon_{\rho}(\rho_1),\varepsilon_{\rho}(\rho_2)\}< \varepsilon_{\rho}^*<\varepsilon_{\rho}(\hat{\rho})$ and let $(\rho_3,\rho_4)$ be the component of $\{\rho\in(\rho_1,\rho_2):\varepsilon_{\rho}(\rho)>\varepsilon_{\rho}^*\}$ containing $\hat{\rho}$. Then $\rho_1<\rho_3<\rho_4<\rho_2$, $\varepsilon_{\rho}(\rho_3)=\varepsilon_{\rho}(\rho_4)=\varepsilon_{\rho}^*$, and:
    \begin{equation*}
        [\varepsilon(\rho_4)-\varepsilon_{\rho}^*\rho_4]-[\varepsilon(\rho_3)-\varepsilon_{\rho}^*\rho_3] =\int_{\rho_3}^{\rho_4}(\varepsilon_{\rho}(\rho)-\varepsilon_{\rho}^*)\,d\rho>0.
    \end{equation*}
    Hence the tangent lines at $\rho_3$ and $\rho_4$ are parallel and distinct, thus disjoint. The reversed-inequality case is analogous.
\end{proof}

\bibliographystyle{elsarticle-num}
\bibliography{cas-refs}

\end{document}